\documentclass{amsart}
\usepackage{amsthm,stmaryrd}
\usepackage{amssymb}
\usepackage{amsmath}
\usepackage{epsfig}
\usepackage[usenames,dvipsnames]{color}
\usepackage{verbatim}
\usepackage{hyperref}
\usepackage{mathrsfs}
\usepackage[all]{xy}
\usepackage{xcolor}
\usepackage{amssymb,amsmath,amsthm}
\usepackage{tikz}
\usetikzlibrary{shapes,snakes}
\usepackage[utf8]{inputenc}
\usepackage{pst-node}

\usetikzlibrary{arrows,calc,positioning,decorations.pathreplacing}
\usepackage[left=3cm,right=3cm,top=3cm,bottom=3cm]{geometry} 

\newcommand{\Z}{\mathbb{Z}}
\newcommand{\Q}{\mathbb{Q}}
\newcommand{\R}{\mathbb{R}}
\newcommand{\C}{\mathbb{C}}
\newcommand{\N}{\mathbb{N}}
\newtheorem{definition}{Definition}
\newtheorem{conjecture}{Conjecture}
\newtheorem{theorem}[definition]{Theorem}

\newtheorem*{conj}{Conjecture}
\newtheorem*{theor}{Theorem}

\newtheorem{notation}{Notation}
\newtheorem{proposition}[definition]{Proposition}
\newtheorem{lemma}[definition]{Lemma}
\newtheorem{remark}[definition]{Remark}

\newtheorem{corollary}[definition]{Corollary}
\newcounter{exo} \newcounter{numexercice}
\renewcommand{\theexo}{\arabic{exo}}

\newcounter{IntroCounter}
\stepcounter{IntroCounter}

\begin{document}
\title[]{A combinatorial description of the centralizer algebras connected to the Links-Gould Invariant} 
\author{Cristina Ana-Maria Anghel}
\address{Mathematical Institute University of Oxford, Oxford, United Kingdom} \email{palmeranghel@maths.ox.ac.uk; cristina.anghel@imj-prg.fr } 
\thanks{}
\date{\today}

\begin{abstract}
In this paper we study the tensor powers of the standard representation of the quantum super-algebra $U_q(sl(2|1)$, focusing on the rings of its algebra endomorphisms, called centraliser algebras and denoted by $LG_n$. Their dimensions were conjectured by I. Marin and E. Wagner \cite{MW}. We prove this conjecture, describing the intertwiner spaces from a semi-simple decomposition as sets consisting of certain paths in a planar lattice with integer coordinates. Using this model, we present a matrix unit basis for the centraliser algebra $LG_n$, by means of closed curves in the plane, which are included in the lattice with integer coordinates.
\end{abstract}

\maketitle
\setcounter{tocdepth}{1}
 \tableofcontents
{
\section{Introduction}
The Links-Gould invariant is a 2-variable link polynomial $LG(L;t_0,t_1) \in \Z[t_0^{\pm 1}, t_1^{\pm 1}]$, constructed from the representation theory of the super-quantum group $U_q(sl(2|1))$. This invariant was introduced by Links and Gould {\cite{LG}} in 1992, and it is a renormalised type invariant for links. In 1991, Reshetikhin-Turaev defined a recipe that having as input a category of representations of a (super) quantum group, leads to quantum invariants for links.
If one starts with a category of representations of a super-quantum group in this context, the Reshetikhin-Turaev method gives vanishing invariants. The problem is the fact that this construction uses the so called quantum dimension of a representation. For super-quantum groups with $q$ parameter (for example $U_q(sl(2|1))$), the quantum dimension of a generic representation is zero. From this reason,
 if one applies this construction to this kind of categories, one gets vanishing invariants. The procedure of renormalisation comes in order to correct this issue, and obtain secondary invariants which do not vanish. The idea is to start with a link, use the Reshetikhin-Turaev method on the link where one strand that is cut, and correct this in such a manner such that it leads to a well defined invariant. Later on, in 2007, N. Geer and B. Patureau introduced a more general method and they constructed a sequence of renormalised quantum invariants from the representation theory of the super-quantum group $U_q(sl(m|n))$. These invariants recover via a certain specialisation of the variables the Links-Gould invariant, for the case of the super-quantum group $U_q(sl(2|1))$. 

The representation theory of $U_q(sl(2|1))$ is rich and has a continous family of representations. For any $n \in \mathbb N$ and $\alpha \in \C$, there is a corresponding $U_q(sl(2|1))$-representation of dimension $4(n+1)$, which will be denoted by $V(n, \alpha)$. The Links-Gould invariant is defined starting from the simple representation $V(0,\alpha)$, corresponding to a generic complex parameter $\alpha \in \C$, in a renormalised Reshetikhin-Turaev type construction. 

Coming from a different direction, the notion of centraliser algebra is well-known in representation theory and refers to the study of intertwiner spaces corresponding to the tower of the tensor powers of certain representations. For example, if we fix $H$ to be a Hopf algebra and $V \in \mathscr Rep(H)$, then the tensor powers of $V$ have again a module structure over $H$. Furthermore, one can construct a tower of algebras, called "centraliser algebras" as:
$$C_n:=End_H(V^{\otimes n}).$$ 
As a particular case, one can study the algebra $H=\mathscr U(g)$, which is the enveloping algebra of a Lie algebra $g$, or the quantisation, $H_q=\mathscr U_q(g) $ which is called the quantum enveloping algebra.
For the case of the Lie algebra $sl(N)$ and its standard representation $V \in \mathscr Rep(\mathscr U(sl(N)))$, at the classical level, one gets the following connection to the group algebra of the symmetric group $S_n$:
$$k[S_n]\twoheadrightarrow End_{\mathscr U(sl(N))}(V^{\otimes n}).$$ At the quantised level, this corresponds to the Hecke algebra:
$$H_n\twoheadrightarrow End_{\mathscr U_q(sl(N))}(V^{\otimes n}).$$
Wenzl studied \cite{Wz} the standard representation $V$ of $so(N)$ and its relation at the classical level to the Brauer algebra $Br_n$:
$$Br_n\twoheadrightarrow End_{\mathscr U(so(N))}(V^{\otimes n}).$$
At the deformed level, Birman and Wenzl (\cite{BW}) showed that the quantisation of the Brauer algebra corresponds to the Birman-Murakami-Wenzl algebra:
$$BMW_n\twoheadrightarrow End_{\mathscr U_q(so(N))}(V^{\otimes n}).$$

In \cite{Mu2}, Kosuda and Murakami described the structure and dimensions of centralizer algebras corresponding to mixed tensor powers of different representations of $U_q(gl(n, \C))$.
Moreover, the dimensions of various types of Birman-Murakami-Wenzl and Brauer algebras were computed in \cite{C1}, \cite{C2}, \cite{C3}, \cite{C4}. 

The study of centraliser algebras for various representations of quantum groups has been broadly developed and it created bridges and interactions between different ways of describing invariants for links. The Birman-Murakami-Wenzl alebras $\{ BMW_n \}_{n \in \mathbb N}$ are a sequence of algebras which are defined as quotients of the group algebra of the braid group by cubic relations.
Murakami showed that these algebras lead to the Kauffman invariant for links (\cite{Mu}, \cite{Mu2}). 
 
Concerning the theory of super-quantum groups, for each complex parameter $\alpha \in \C$, the corresponding representation $V(0, \alpha)$ of $U_q(sl(2|1))$ will lead to a sequence of centraliser algebras $\{LG_n\}_{n \in \N}$, denoted by:
$$LG_n(\alpha)=End_{U_q(sl(2|1))}(V(0,\alpha)^{\otimes n}).$$  The super quantum group $U_q(sl(2|1))$ has an $R$-matrix, that leads to a braid group representation:
$$\rho_n(\alpha): \Bbbk B_n \rightarrow LG_n(\alpha).$$
We are interested to study properties related to this morphism. Our motivation for this question comes from its close relation to the Links-Gould invariant for links. 
In 2011, Marin and Wagner (\cite {MW}) proved that this morphism in surjective and conjectured the dimension of this centraliser algebra:
\begin{theorem}(Marin-Wagner Conjecture \cite{MW})\label{conj}
$$dim(LG_{n+1})(\alpha)=\frac{(2n)!(2n+1)!}{(n!(n+1)!)^2}.$$
\end{theorem}
The main result of this paper is the proof if this conjecture. We use combinatorial techniques in order to encode the semi-simple decompositions of tensor powers of the representation $V(0,\alpha)$. 

The first step consists in a mixture of algebraic and combinatorial tools, permitting to pass from the initial algebraic question, concerning the computation of this dimension of the algebra, to a purely combinatorial problem.
More specifically, we construct certain diagrams in the lattice with integer coordinates on the plane, where we assign to each point a certain weight (which is a natural number). This encodes the dimension of a corresponding intertwiner space in $V(0,\alpha)^{\otimes n}$. Inductively, we show that the intertwiner spaces can be described using a set of paths in the plane with prescribed moves. 
 
Let us consider the following set of planar moves:
$$M1) \ \ \ \ \ \ \ \text{stay  move} \ \ \ \ \ (x',y')\rightarrow (x',y') \ \ \ \ \ $$
$$M2) \ \ \ \ \ \  \ \longrightarrow  \ \ \ \ \ \ \ \ \ \ \ \ (x',y')\rightarrow (x'+1,y')$$
$$M3) \ \ \ \ \ \ \ \ \ \uparrow  \ \ \ \ \ \ \ \ \ \ \  \ \ (x',y')\rightarrow (x',y'+1) $$
$$ \ \ \ \ \ \ \ \ \ \ \ \ \ \ \ M4)  \ \ \ \ \ \ \ \ \nwarrow  \ \ \ \ \ \ \ \ \ \ \ \ (x',y')\rightarrow (x'-1,y'+1) \text{ if } x'>0.$$
The main point of this part leads to the following combinatorial description for the isotypic components that occur in the decomposition of the tensor powers of $V(0,\alpha)$.
\begin{theorem}(Combinatorial model for the isotypic components corresponding to $V(0,\alpha)^{\otimes n}$)\label{Thm1}

Let $\alpha \in \C \setminus \Q $. For any $n \in \N$, let us consider the semi-simple decomposition:
$$V(0,\alpha)^{\otimes n}=\bigoplus_{x,y\in \N\times \N} \ \left( \ T_n(x,y)\otimes V(x,k\alpha+y) \ \right).$$ Here, $T_n(x,y)$ is the multiplicity space  corresponding to the weight $(x,n\alpha+y) \in \N \times \C$.\\  
Then, there is the following one to one correspondence:
$$ T_{n+1}(x,y) \longleftrightarrow \{ \sigma  \subseteq \R^2 \mid \sigma \text{ is a path of length n, in the first quadrant }  $$
$$ \ \ \ \ \ \ \ \ \ \ \ \ \ \ \ \  \text{ between the points } (0,0) \text{ to } (x,y) $$
$$ \text{ \ \ \ \  \ \ \ \ \ \ \ \ \ \ \ \ \ \ \ \ \ \ \ \ \ \ \ \ \ \ \ \ \ constructed with the moves } M_1, M_2, M_3 \text{ or} \  M_4 \}. $$
\end{theorem}

On the other hand, it was previously known from combinatorics that the conjectured dimension $$\frac{(2n)!(2n+1)!}{(n!(n+1)!)^2}$$ is the result of a counting problem of certain pairs of paths in the plane.
The second part of the paper concerns a combinatorial problem, where we make the correspondence between two sets of planar paths. We establish a bijection between the previous set of paths-that characterise the intertwiner spaces and this set of pairs of paths, that occur in the description of the right hand side of the conjectured formula. This interpretation leads to the desired conjectured dimension and emphasises combinatorial properties of the image of the morphism $\rho_n(\alpha)$. 

In their paper, Marin and Wagner investigated further the morphism $\rho_n(\alpha)$. They showed that it factors through a cubic Hecke algebra denoted by $H(\alpha)$. Moreover, they considered certain relations that are in the kernel of this map: $r_2$ for three strands and $r_3$ for the braid group with four strands. In this way, they defined a smaller quotient of the cubic Hecke algebra $A_n(\alpha)$, and they conjectured that this is enough in order to generate the kernel of $\rho_n(\alpha)$:
\begin{conjecture}(Marin-Wagner \cite{MW})\label{conj2}
 Consider the following quotient of the cubic Hecke algebra:
 $$A_n(\alpha):=H_n(\alpha)/(r_2,r_3).$$
 Then, for any $n \in \N$, there is the isomorphism of algebras:
 $$A_n(\alpha)\simeq LG_n(\alpha).$$  
 \end{conjecture}

 On the representation theory part of the story, the question of finding matrix unit bases for centraliser algebras has been studied for some important algebras and it is related to the theory of quantum invariants for knots. On this subject, Wenzl (\cite{Wz}) and Ram and Wenzl (\cite{RW}) described a matrix unit basis for Brauer's centraliser algebras and for the Hecke algebras of type A. Moreover, in \cite {BB}, Blanchet and Beliakova described a precise basis of matrix units for the Birman-Murakami-Wenzl algebra, using idempotent elements and skein theory.  In 2006, Lehrer and Zhang (\cite{LeZa}) studied the cases where the morphism obtained from the group algebra of the braid group onto the automorphism group of the tensor power of a certain representation, given by infinitesimal actions is surjective.

 Related to this question, in the last section of the paper (section \ref{Sec7}), we describe a matrix unit basis for the centraliser algebra $LG_n$.
This model is established using Theorem \ref{Thm1} and gives combinatorial counterparts for the projectors onto the simple components that occur in the semi-simple decompositions of tensor powers of $V(0, \alpha)$.
 
 \begin{notation} 
 
 Consider $\Delta_n \subseteq \R^{2}$ to be the standard simplex with edge of length $n$.\\
Let us denote by: 
$$S_{n+1}^{ref}(a,b):= \{ \sigma \subseteq \R^2 \mid \sigma \text{ is a simple closed curve containing } \left(0,0\right) \text{ and } \left(n+1,n+1\right), $$
$$ \ \ \ \ \ \ \ \ \ \ \ \ \ \ \ \ \ \ \  \ \ \ \ \ \ \ \ \ \ \text{ constructed with the moves } M_2 \text{ or } M_3, \text{ which is symmetric}$$
$$ \ \ \ \ \ \ \ \ \ \ \ \ \ \ \ \ \ \ \ \ \ \text{with respect to the secondary diagonal,}$$
$$ \ \ \ \ \ \ \ \ \ \ \ \ \ \ \ \ \ \ \ \ \ \ \ \ \ \ \ \ \ \ \ \ \ \  \  \text{and cut it into two points:} { \ (n+1-a,a) \text{ and }  (n+1-b,b) \} }.$$
\end{notation}

\begin{theorem}\label{Thm3}

1) (Matrix unit basis for $LG_{n+1}(\alpha)$)

There is an algebra isomorphism which makes the correspondence between a basis of projectors onto simple components from $LG_{n+1}(\alpha)$, and an algebra of elementary matrices indexed by symmetric planar curves:
$$LG_{n+1}(\alpha) \ \simeq \ < \sigma \subseteq \R^2 \mid \sigma \text{ is a simple closed curve}$$
$$ \ \ \  \ \ \ \ \ \  \ \ \ \ \ \ \ \ \ \ \ \  \ \ \ \ \ \ \ \ \ \ \ \ \ \ \ \ \ \ \ \ \ \text { containing } \left(0,0\right) \text{ and } \ \left(n+1,n+1\right) $$
$$ \ \ \ \ \ \ \ \ \ \ \ \ \ \ \ \ \ \ \ \ \ \ \ \ \ \ \ \ \ \ \ \ \ \ \ \ \ \ \ \ \ \ \ \ \ \ \ \ \ \ \ {\text{ constructed with the moves } M_2 \text{ or } M_3 >_{\C}}^{\ast} .$$
2) (Identification of projectors onto isotypic components ) 

For any natural number $n \in \N$ and any point $(x,y)\in \Delta_n$, we denote by $p^{n+1}_{x,y}\in LG_{n+1}(\alpha)$ the projector onto the isotypic component of $V\big(x,(n+1)\alpha+y\big)$ inside $V(0, \alpha)^{n+1}$. 

Then, there is a one to one correspondence between this set of projectors and the following family of symmetric planar curves, defined as follows:
$$p^{n+1}_{x,y} \ \ \ \ \ \ \ \longleftrightarrow \sum_{\sigma \in S_{n+1}^{ref}(x+y+1,y)} \sigma^{\star}, \ \ \ \ \ \ \ \ \ \  \forall (x,y)\in \Delta_n.$$

\end{theorem}  

Our aim for a future paper is to use this combinatorial model for a matrix unit basis for the Links-Gould centraliser algebra $LG_n$, in order understand the morphism $\rho_n(\alpha)$.


\

 {\bf{Structure of the paper}}:

This paper is split into six main parts. In Section \ref{sec2}, we present the super quantum group $U_q(sl(2|1))$ and some properties about its representation theory.
Then, in Section \ref{sec3} we review the definition of the Links-Gould invariant. Section \ref{sec4} is devoted to the definition of the centraliser algebras $LG_n(\alpha)$ and certain conjectures about them. Then, in Section \ref{sec5}, we discuss a combinatorial model for the isotypic components of $V(0,\alpha)^{n}$. Section \ref{sec6} is devoted to the proof of the Marin-Wagner conjecture about the dimensions of these centraliser algebras.  In section \ref{sec7'} we put together all the steps of the proof and summarise a concrete algorithm that makes the connection between multiplicity spaces in $V(0,\alpha)^{\otimes n}$ and pairs of paths in the simplex. In the last part, Section \ref{Sec7}, we give a combinatorial matrix unit basis for the algebra $LG_{n}$.  

\

{\bf{Acknowledgements}}: 
 I would like thank very much to my advisor, Professor Christian Blanchet for suggesting this problem and for many discussions, references and helpful advice. Also, I want to thank Professor Frédéric Chapoton for pointing me the references \cite{A} and \cite{S}.
The first version of this paper, prepared during my PhD, was supported by grants from Région Ile-de-France. For this second version, I acknowledge the support from the European Research Council (ERC) under the European Unions Horizon 2020 research and innovation programme (grant agreement No 674978).
\section{The quantum group $U_q(sl(2|1))$}\label{sec2}
In this section we will introduce the super-quantum group $U_q(sl(2|1))$ and discuss about its representation theory. Let us fix a ground ring $\Bbbk$.  A super vector space is a vector space over $\Bbbk$ with a $\Z_2$ grading: $V=V_{0} \oplus V_{1}$. A homogenous element $x \in V$ is called even if $x \in V_{0}$ and odd if $x \in V_{1}$. Through this section, all the objects that we will work with will respect the $\Z_2$-gradings.
\begin{notation} Consider $\hslash$ to be an indeterminate and the field $\Bbbk:=\C ((\hslash))$. Denote by:
$$q:=e^{\frac{\hslash}{2}} \ \ \ \ \ \  \ \ \ \ \ \ \ \ \ \ \{x\}:=q^{x}-q^{-x}$$ 
$$exp_q(x):= \sum_{n=0}^{\infty} \frac{x^n}{(n)_q!} \ \ \ \ \ \ \ \ \ \ \ \ (k)_q=\frac{1-q^k}{1-q} \ \ \ \ \ \ \ \ \ \ (n)_q!=(1)_q (2)_q...(n)_q.$$
For two homogeneous elements $x$ and $y$ with gradings $\bar{x}, \bar{y}\in \Z_2$, the super-commutator has the formula:
$$[x,y]:=xy- (-1)^{\bar{x}\bar{y}}yx.$$  

\end{notation} 
\begin{definition}
Let $A=(a_{ij})$ be the square matrix given by
$A=
\begin{pmatrix} 
 \ 2  &-1\\
-1 & 0 
\end{pmatrix}
$.\\
Consider $U_q(sl(2|1))$ to be the superalgebra over $\C((\hslash))$ generated by 
$\{E_i,F_i,h_i\}_{i \in \{ 1,2\} }$ where $E_2$ and $F_2$ are odd and all the others generators are even, with the following relations:\\  
\begin{align*}
 [h_i,h_j]&=0 , &  [h_i,E_j]=& a_{i,j} E_j  , &  [h_i,F_j]=& -a_{i,j} F_j,
\end{align*}
\begin{align*}
[E_{i},F_{j}]=&\delta_{i,j}\frac{q^{h_i}-q^{-h_i}}{q-q^{-1}}, & E_2^2=&F_2^2=0
 \end{align*}
\begin{align}
\label{E:QserreA2}
  E_{1}^{2}E_{2}-(q+q^{-1})E_{1}E_{2}E_{1}+E_{2}E_{1}^{2}&=0, &
  F_{1}^{2}F_{2}-(q+q^{-1})F_{1}F_{2}F_{1}+F_{2}F_{1}^{2}&=0.
\end{align}

The algebra $U_q(sl(2|1))$ is a Hopf algebra where the coproduct, counit and
antipode are defined by
\begin{align*}
 \Delta(E_{i})= & E_{i}\otimes 1+ q^{-h_i} \otimes E_{i}, &   \epsilon(E_{i})= & 0 & S(E_{i})=&-q^{h_{i}}E_{i}\\
 \Delta (F_{i}) = & F_{i}\otimes q^{h_{i}}+ 1 \otimes F_{i}, &   \epsilon(F_{i})= &0 & S(F_{i})=&-F_{i} q^{-h_{i}}\\
   \Delta(h_i) = & h_{i} \otimes 1 + 1\otimes h_{i}, & \epsilon(h_{i})  = & 0 & S(h_{i})= &-h_{i}.
 \end{align*}
\end{definition} 
\begin{notation}
Consider the following elements:
$$E':= E_1E_2-q^{-1}E_2E_1 \ \ \ \ \ F'=F_2F_1-qF_1F_2.$$
\end{notation}

\begin{proposition}{\label{Rmatrix}}
In \cite{KT},\cite{Y} it has been shown that the quantum group $U_q(sl(2|1))$ admits a universal $R$-matrix $R \in U_q(sl(2|1)) \hat{\otimes}U_q(sl(2|1))$.
Consider the following expressions:  
$$\check{R}=exp_q \left( \{ 1 \} E_1 \otimes F_1 \right) exp_q \left( - \{ 1 \} E' \otimes F' \right) exp_q \left( - \{ 1 \} E_2 \otimes F_2 \right) $$
$$K=q^{-h_1 \otimes h_2-h_2 \otimes h_1-2 h_2 \otimes h_2}.$$
Then $R=\check{R}K$.
\end{proposition}

In this sequel we will concentrate on the representation theory of the super-quantum group $U_q(sl(2|1))$. We would like to emphasyse that the representation theory of super-quantum groups has continuous families of representations already when $q$ is generic. This is in contrast with the classical case of quantum groups, where it is possible to gain a continuous family of representations, just if $q$ is specialised to a root of unity. We will follow \cite{GP}.
\begin{definition}
1) An element $v \in V$ is called a weight vector of weight $\lambda=(\alpha_1,\alpha_2)$ if:
$$h_iv=\alpha_iv, \forall \ i\in \{1,2  \}.$$
2) A weight vector is called highest weight vector if:
$$E_i v =0, \forall i \in \{1,2\}. $$
3)A module $V$ is called a highest weight module of weight $\lambda=(\alpha_1,\alpha_2)$ if it is generated by a highest weight vector $v_0 \in V$ of weight $\lambda$.
\end{definition}
\begin{proposition}(Simple $U_q(sl(2|1))$ representations)

There exists a continuous family of simple representations of $U_q(sl(2|1))$,\\
indexed by $\Lambda=\N \times \C$:
$$ \lambda=(\alpha_1,\alpha_2) \in \Lambda \longleftrightarrow V(\alpha_1,\alpha_2) \text{ highest weight module of weight} \ \lambda$$ 
The weight $\lambda$ is called ``typical'' if $ \alpha_1+\alpha_2 \neq -1$ and $\alpha_2\neq 0$, otherwise it is called ``atypical''. 
Using this, the previous family splits into two types of representations. A highest weight module $V(\alpha_1,\alpha_2)$ is called typical/ atypical if its corresponding highest weight is typical/ atypical.
\end{proposition}

Notation: For any $V,W \in Rep(U_q(sl(2|1)))$, we denote by $$R_{V,W}\in  Aut_{U_q(sl(2|1))} \left( V \otimes W\right)$$ the isomorphism obtained by the action of $R$ (\ref{Rmatrix}) onto $V \otimes W$.

\begin{remark}(Braiding on the category of representations)

 The $R$-matrix will induce a braiding on the category of finite dimensional $U_q(sl(2|1))$-representations. For any $V,W \in Rep(U_q(sl(2|1)))$, it leads to an isomorphism $\mathscr{R}_{V,W}: V \otimes W \rightarrow W \otimes V$ defined in the following manner:
$$\mathscr{R}_{V,W}:=\tau^s \circ R_{V \otimes W}.$$
Here, $\tau^s: V \otimes W \rightarrow W \otimes V$ is the called super flip. If $x \in V$ and $y \in W$ are two homogeneous elements, then the super flip acts in the following way:
$$\tau^s(x \otimes y)=(-1)^{(deg x)(deg y)} y \otimes x.$$ 
\end{remark}
\begin{proposition}\label{braid}(Braid group action)

The braiding from the category of $U_q(sl(2|1))$-representations induces a well defined braid group action on tensor powers of $V(\lambda)$ for any $\lambda \in \Lambda$:
$$\rho_n: B_n \rightarrow Aut_{U_q(sl(2|1))}(V(\lambda)^{\otimes n})$$
$$\rho_n(\sigma_i^{\pm 1})=Id^{\otimes i-1} \otimes \mathscr{R}^{\pm 1}_{V(\lambda),V(\lambda)}\otimes Id^{n-i-1}.$$
\end{proposition}
\begin{theorem}\label{Thm}(Decomposition of the tensor products \cite{GP}, $Lemma \ 1.3$)

If $\alpha ,\beta \in \C^{\ast }, n\in \N$ such that all the modules from the following expression are typical, then $V(0,\alpha)\otimes V(n,\beta)$ is semi-simple and has the following decomposition:\\
1) For $n\not=0$:
$$V(0,\alpha)\otimes V(n,\beta)=V(n,\alpha+\beta)\oplus V(n+1,\alpha+\beta)\oplus V(n-1,\alpha+\beta+1)\oplus V(n,\alpha+\beta+1).$$
2) For $n=0$: $$V(0,\alpha )\otimes V(0,\beta )=V(0,\alpha+\beta)\oplus  V(0,\alpha+\beta+1)\oplus V(1,\alpha+\beta).$$
\end{theorem}
\section{The Links-Gould invariant}\label{sec3}
 In this section, we will recall the definition of the Links-Gould invariant for links. After that, in the second part, we will see how this invariant can be recovered from a more general set of renormalised invariants introduced by Geer and Patureau, using the representation theory of the super quantum group $U_q(sl(m|n))$. 
 
Let $K:=\C(t^{\pm \frac{1}{2}}_0,t^{\pm \frac{1}{2}}_1)$ and $V:=<v_1,...,v_4>$ a $K$-vector space of dimension 4. 

Consider the set $\mathscr B$ to be the following ordered basis for $V$:
$$\mathscr B:=(v_1\otimes v_1,...,v_1 \otimes v_4,v_2\otimes v_1,...,v_2 \otimes v_4,...,v_4 \otimes v_1,...,v_4 \otimes v_4  ).$$
Consider the operator $R \in Aut(V \otimes V)$ defined in the basis $\mathscr B$ as follows:\\
$$R=\left[
\begin{array}{cccccccccccccccc}
 t_0 & \cdot & \cdot & \cdot & \cdot & \cdot & \cdot &  \cdot & \cdot & \cdot & \cdot & \cdot & \cdot & \cdot & \cdot & \cdot \\
\cdot & \cdot & \cdot & \cdot & t^{\frac{1}{2}}_0 & \cdot & \cdot &  \cdot & \cdot & \cdot & \cdot & \cdot & \cdot & \cdot & \cdot & \cdot\\
\cdot & \cdot & \cdot & \cdot & \cdot & \cdot & \cdot &  \cdot & t^{\frac{1}{2}}_0  & \cdot & \cdot & \cdot & \cdot & \cdot & \cdot & \cdot\\
\cdot & \cdot & \cdot & \cdot & \cdot & \cdot & \cdot &  \cdot & \cdot  & \cdot & \cdot & \cdot & 1 & \cdot & \cdot & \cdot\\
\cdot & t^{\frac{1}{2}}_0 & \cdot & \cdot &  t_0-1 &  \cdot & \cdot &  \cdot & \cdot & \cdot & \cdot & \cdot & \cdot & \cdot & \cdot & \cdot \\
\cdot & \cdot & \cdot & \cdot & \cdot & -1 & \cdot &  \cdot & \cdot & \cdot & \cdot & \cdot & \cdot & \cdot & \cdot & \cdot \\
\cdot & \cdot & \cdot & \cdot & \cdot & \cdot & t_0t_1-1 &  \cdot & \cdot & -t^{\frac{1}{2}}_0t^{\frac{1}{2}}_1 & \cdot & \cdot & -t^{\frac{1}{2}}_0t^{\frac{1}{2}}_1Y & \cdot & \cdot & \cdot \\
\cdot & \cdot & \cdot & \cdot & \cdot & \cdot & \cdot &  \cdot & \cdot & \cdot & \cdot & \cdot & \cdot & t^{\frac{1}{2}}_1 & \cdot & \cdot \\
\cdot & \cdot & t^{\frac{1}{2}}_0 & \cdot &  \cdot &  \cdot & \cdot &  \cdot & t_0-1 & \cdot & \cdot & \cdot & \cdot & \cdot & \cdot & \cdot \\
\cdot & \cdot & \cdot & \cdot & \cdot & \cdot & -t^{\frac{1}{2}}_0t^{\frac{1}{2}}_1 &  \cdot & \cdot & \cdot & \cdot & \cdot & Y & \cdot & \cdot & \cdot \\
 \cdot & \cdot & \cdot & \cdot & \cdot & \cdot & \cdot &  \cdot & \cdot & \cdot & -1 & \cdot & \cdot & \cdot & \cdot & \cdot \\
 \cdot& \cdot & \cdot & \cdot & \cdot & \cdot & \cdot &  \cdot & \cdot & \cdot & \cdot & \cdot & \cdot & \cdot & t^{\frac{1}{2}}_1 & \cdot \\
 \cdot & \cdot & \cdot & 1 & \cdot & \cdot & -t^{\frac{1}{2}}_0t^{\frac{1}{2}}_1Y &  \cdot & \cdot & Y & \cdot & \cdot & Y^2 & \cdot & \cdot & \cdot \\
 \cdot & \cdot & \cdot & \cdot & \cdot & \cdot & \cdot & t^{\frac{1}{2}}_1 &  \cdot  & \cdot & \cdot & \cdot & \cdot & t_1-1 & \cdot & \cdot\\
  \cdot & \cdot & \cdot & \cdot & \cdot & \cdot & \cdot & \cdot &  \cdot  & \cdot & \cdot & t^{\frac{1}{2}}_1 & \cdot & \cdot & t_1-1 & \cdot\\
 \cdot & \cdot & \cdot & \cdot & \cdot & \cdot & \cdot &  \cdot & \cdot & \cdot & \cdot & \cdot & \cdot & \cdot & \cdot & t_1 \\
\end{array}
\right]$$
Here, Y has the following expression:  $$Y=\left((t_0-1)(1-t_1)\right)^{\frac{1}{2}}.$$
\begin{proposition}\cite{DKL}
The operator $R\curvearrowright V \otimes V$ satisfies the Yang-Baxter equation, and it induces a sequence of braid group representations:
$$\varphi_n: B_n \rightarrow Aut(V^{\otimes{n}})$$
$$\varphi(\sigma^{\pm 1}_i)= Id^{\otimes i-1}\otimes R^{\pm 1} \otimes Id^{\otimes n-i-1}.$$
\end{proposition}
\begin{definition}
Consider the operator $\mu \in Aut(V)$ defined in the basis $\{v_1,...,v_4 \}$ by:
$$\mu=\left[
\begin{array}{*{16}c}
 t^{-1}_0 & \cdot & \cdot & \cdot &  \\
 \cdot & -t_1 & \cdot & \cdot &  \\
\cdot & \cdot & -t_0^{-1} & \cdot &  \\
\cdot & \cdot & \cdot & t_1 &  \\
\end{array}
\right].$$
\end{definition}
\begin{notation}
Consider $W$ to be a vector space over $K$. Let $f \in End(W)$ to be an endomorphism which is a scalar times the identity. We denote by $<f>$ the corresponding element of K:
$$f=<f> Id_W.$$ 
\end{notation}
\begin{theorem}(The Links-Gould invariant)

For any braid $\beta \in B_n$, the following partial trace leads to a scalar:
$$Tr_{2,...,n}((Id_V \otimes \mu^{\otimes(n-1)})\circ \varphi_n(\beta))\in K\cdot Id_V.$$
The Links-Gould polynomial is defined using this partial trace in the following manner:
$$LG(L;t_0,t_1):=<Tr_{2,...,n}((Id_V \otimes \mu^{n-1})>.$$
Then, this is a well defined link invariant and it has integer coefficients (\cite{I2}):
$$LG(L;t_0,t_1)\in \Z[t^{\pm 1}_0,t^{\pm 1}_1].$$
\end{theorem}
Now, we present the Geer-Patureau method, that leads to modified quantum invariants for links starting from representation theory of super quantum groups. 
As we have seen, the category of representations of $U_q(sl(2|1))$ has a brading, which leads to a braid group representation. Moreover, it is known that it has a structure of a ribbon category \cite{GP1}. 
\begin{definition}
Let $\mathscr C$ be a category. The category of $\mathscr C$-colored framed tangles $\mathscr T_{\mathscr C}$ is defined as follows:
$$Ob(\mathscr T_{\mathscr C})= \{(V_1,\epsilon_1),...,(V_m,\epsilon_m) | m \in \mathbb N,\epsilon_i \in {\pm1}, V_i \in  \mathscr C \}.$$ 
$$Hom_{\mathscr T_{\mathscr C}}\left( (V_1,\epsilon_1),...,(V_m,\epsilon_m);(W_1,\delta_1),...,(W_n,\delta_n)\right) = \{ \mathscr C- \text{colored framed tangles}$$
$$ \ \ \ \ \ \ \ \ \ \ \ \ \ \ \ \ \ \ \ \ \ \ \ \ \ \ \ \ \ \ \ \ \ \ \ \ \ \ \ \ \  T : (V_1, \epsilon_1), ..., (V_m, \epsilon_m) \uparrow (W_1, \delta_1), ..., (W_n, \delta_n)/ \text{isotopy}.$$
Remark: The tangles  $T$ have to respect the boundary colors $ \{V_i, W_j \}$.
Once we have such a tangle, it has an induced orientation, coming from the signs $\epsilon_i$ , using the following conventions:
$$(V,-) \ \downarrow,  \ \ \ (V,+) \ \uparrow.$$
\end{definition}
\begin{theorem}(Reshetikhin-Turaev)
Let $\mathscr C$ be a ribbon category. Then, there exist a unique monoidal functor 
$$\mathscr F: \mathscr T_{\mathscr C}\rightarrow \mathscr C$$
such that 
$\forall V \in \mathscr C$, it respects the following relations:
$$ \ \mathscr F((V,+))=V; \ \ \ \mathscr F((V,-))=V^{*} \ \ \ \ \ \ \ \ \ \ \ \ $$ 
and it is fixed onto elementary morphisms by the ribbon structure of the category $\mathscr C$.
\end{theorem}

This construction applied to the representation theory of $U_q(sl(2|1)$ leads to vanishing invariants. The problem comes from the fact that the quantum dimension of a representation (which corresponds to the evaluation of the Reshetikhin-Turaev functor onto the circle coloured with that representation) vanishes. Then, the Geer-Patureau construction uses the Reshetikhin-Turaev functor on the link with one strand that is cut and correct it in a way that leads to a well defined invariant. They introduce the notion called modified quantum dimension, denoted by $d$, and use that instead of the usual quantum dimension, to close up the last strand that is cut. We present the main steps of the construction.

Let $L$ be a link with $k$ components. Consider $n \in \N$ and $\alpha_1,...,\alpha_k \in \C$. One colors the components of $L$ with the representations $\{V(n,\alpha_1),...,V(n,\alpha_k) \}$. Then, let us choose a strand of $L$, and denote by $L'$ the $(1,1)$-tangle obtained from $L$ by cutting this strand. Let $i \in \{1,...,k   \}$ such that the strand that we cut has the colour $V(n, \alpha_i)$.
\begin{theorem}(\cite{GP1})
1) Consider the following modified construction, defined as:
$$F'(L):=d(V(n, \alpha_i)) <\mathscr F (L')>.$$ 
Then, $F'$ is a well defined link invariant, called the Geer-Patureau modified invariant. 

2) There exist a link invariant, which is a polynomial in $k+1$ variables 
$$M^n(L)(q,q^{\alpha_1},...,q^{\alpha_k})\in \Q(q,q_1,...,q_k)$$
such that 
$$F'(L)=e^{- \sum lk_{i,j}(2 \alpha_i \alpha_j+ \alpha_i + \alpha_j)} \cdot M^n(L)(q,q^{\alpha_1},...,q^{\alpha_k}).$$
(here $lk_{i,j}$ is the linking number between the strands $i$ and $j$).
\end{theorem}
\begin{theorem}(\cite{GP1},\cite{K})
The modified Geer-Patureau invariants constructed from $U_q(sl(2|1))$ recover the Links-Gould invariant, by a specialisation of coefficients:
$$LG(L;t_0,t_1)|_{(t_0=q^{-2 \alpha},t_1=q^{2\alpha+2})}=\{\alpha\}\{\alpha+1\}M^{0}(q,q^{\alpha},...,q^{\alpha})$$

\end{theorem}
\section{The centralizer algebra $LG_n$}{\label{sec4}}
In this section, we will introduce a family of centraliser algebras corresponding to a sequence of tensor powers associated to a fixed representation $V(0,\alpha)$ of the super-quantum group $U_q(sl(2|1))$. The aim is to understand this family and its relation to the group algebra of the braid group.
\begin{definition}
Let us fix $\alpha \in \C\setminus \Q$. The centraliser algebra corresponding to the representation $V(0,\alpha)$ is defined as:
$$LG_n(\alpha):=End_{U_q(sl(2|1))}\left(  V(0, \alpha)^{\otimes n} \right) $$ 
Then $\{LG_n\}_{n \in \N}$ leads to a sequence of algebras, such that each of them is included in a natural way into the next one:
$$LG_{n-1}(\alpha)\subseteq LG_{n}(\alpha)$$ 
Using these inclusions, $LG_{n}(\alpha)$ becomes a bimodule over $LG_{n-1}(\alpha)$.
\end{definition}
\begin{remark}
The braid group action on $V(0,\alpha)^{\otimes n}$ (defined in Proposition \ref{braid}), extends to the following morphism, at the level of the group algebra of the braid group:
$$\rho_n(\alpha): \Bbbk B_n \rightarrow LG_n(\alpha)$$
\end{remark}
\begin{theorem}(Marin-Wagner \cite{MW})
The morphism $\rho_n(\alpha)$ is surjective.
\end{theorem}
Once we know that this morphism is surjective, an interesting question that arises is to study more deeply the kernel of this map. Marin and Wagner characterised these subalgebras for small values of $n$ ($n \leq 5$) and stated a couple of conjectures about them. Firstly, the question would be to compute the dimension of $LG_n(\alpha)$. Further on, one could wonder which is a set of generating relations for the kernel of the map $\rho_n(\alpha)$. In other words, the question would be to find a set of relations which are in the kernel of this map, such that if we consider the quotient of the algebra $\Bbbk B_n$ by these relations, the result becomes isomorphic to $LG_n(\alpha)$. In the sequel we will present in more details their conjectures.

\

\begin{conj}(Marin-Wagner; Theorem \ref{conj} ) 
$$dim(LG_{n+1})(\alpha)=\frac{(2n)!(2n+1)!}{(n!(n+1)!)^2}.$$
\end{conj}
Moreover, going further in the study of the morphism $\rho_n(\alpha)$, the following part refers to the differences between the algebras $LG_n(\alpha)$ and $\Bbbk B_n$.

\begin{definition}(The cubic Hecke algebra)\\
Let $a,b,c \in \Bbbk^*$. Define the corresponding cubic Hecke algebra as:
$$H_n(a,b,c):=\Bbbk B_n / \left( (\sigma_1-a)(\sigma_1-b)(\sigma_1-c)\right)$$
\end{definition}
Actually, for specific values of the parameters, the cubic Hecke algebra will lead to the centraliser algebra $LG_n$. 
\begin{proposition}(\cite{MW})
Let $\alpha \in \C \setminus \Q$ such that $\{1, \alpha, \alpha^2  \}$ are linearly independent over $\Q$.\\
Consider the following parameters $$a=-q^{-2\alpha(\alpha+1)}, \ \ \ \ \  b=q^{-2\alpha^2}, \ \ \ \ \  c=q^{-2(\alpha+1)^2}.$$
Denote the associated cubic Hecke algebra by:
$$H_n({\alpha})=H_n(a,b,c).$$
Then the morphism $\rho_n(\alpha)$ factors through the algebra $H_n(\alpha)$, leading to a quotient morphism $\tilde{\rho}_n(\alpha)$:
$$ \ \ \ \ \ \ \ \ \ \ \ \ \ \ \ \ \ \ \ \ \ \ \ \ \ \ \ \rho_n(\alpha) \ \ \ \ \ \ \ \ \ \ \ \ \ \ \ \ \ \ \ \ \ \ \ \ \ $$
$$ \ \ \ \ \ \ \mathbb \Bbbk B_n  \ \ \ \ \  \ \longrightarrow \ \ \ \  LG_n(\alpha)$$
$ \ \ \ \ \ \ \ \ \ \ \ \ \ \ \ \ \ \ \ \ \ \ \ \ \ \ \ \ \ \ \ \ \ \ \ \ \ \ \  \ \ \ \ \ \ \ \ \ \ \ \ \ \ \ \ \ \ \searrow \ \ \ \ \ \ \ \ \ \nearrow  \ \tilde{\rho}_n(\alpha)\ \ \ \ \ $

$ \ \ \ \ \ \ \ \ \ \ \ \ \ \ \ \ \ \ \ \ \ \ \ \ \ \ \ \ \ \ \ \ \ \ \ \ \ \ \  \ \ \ \ \ \ \ \ \ \ \ \ \ \ \ \ \ \  H_n(\alpha) \ \ \ \ \ \ \ \ \ \  $
\end{proposition}
In \cite{I}, Ishii introduced a relation $r_2$ in $H_3(\alpha)$ and showed that
this relation is in the kernel of $\tilde{\rho}_3(\alpha)$. However, for the braid group with $4$ strands, one needs more relations in order to generate the kernel of this map. In \cite{MW}, the authors defined a relation $r_3 \in H_4(\alpha)$ and proved that $H_4/ (r_2,r_3)\simeq LG_4$. They conjectured that these relations are enough in order to describe the kernel of $\rho_n(\alpha)$ for all values of $n\in \N$.
\begin{definition}
Consider the following quotient algebra defined as:
$$A_n(\alpha):=H_n(\alpha)/ (r_2,r_3).$$
\end{definition}
\begin{conj}(Marin-Wagner; Conjecture \ref{conj2})
For any number of strands $n \in \N $, there is the following algebra isomorphism:
$$A_n(\alpha) \simeq LG_n(\alpha).$$
\end{conj}
In the following two sections, we prove the Conjecture Marin-Wagner (Theorem \ref{conj}). We summarize all the steps that occur in these sections in diagram \ref{diagr}.

\section{Combinatorial model for the isotypic components corresponding to $V(0,\alpha)^{\otimes n}$} \label{sec5}

This is the first part of the proof and concerns the isotypic components that occur in the decomposition of tensor powers of $V(0, \alpha)^{\otimes n}$. Our aim is to understand them in a combinatorial manner and to prove Theorem\ref{Thm1}.

In order to control the multiplicity spaces of the isotypic components of $V(0,\alpha)^{\otimes n}$, we will encode the semi-simple decomposition of $V(0,\alpha)^{\otimes n}$ by a certain planar diagram $D(n)$. Each point from this diagram will have a weight, which corresponds to a certain multiplicity of a simple representation inside the the $n^{th}$ tensor power of $V(0,\alpha)$.
The next part interpolates between the weights that occur in the diagram $D(n)$, which are algebraic quantities and certain sets of planar paths, described in purely combinatorial terms. More specifically, we show that each weight from the diagram $D(n)$ can be seen as a counting problem associated to a certain set of paths in the plane, which have fixed length and are constructed from the origin using certain prescribed possible moves. As a conclusion we will obtain a model for the multiplicities of the isotypic components in terms of a combinatorial counting problem as in the statement from \ref{Thm1}. 

\

As we have seen
$$LG_n(\alpha)=End_{U_q(sl(2|1))}\left(V(0,\alpha)^{\otimes n}\right).$$
We are interested in the semi-simple decomposition of $V(0,\alpha)^{\otimes n}$. Since $\alpha \notin \Q $, one can conclude by an inductive argument that for any $n \in \N$, the tensor power $V(0,\alpha)^{\otimes n}$ is semi-simple and the formula given in Theorem \ref{Thm} can be applied at each step. More precisely, one remarks that all the modules that occur in the decomposition of $V(0,\alpha)^{\otimes n}$ have the form $$V(x, n\alpha+y)  \ \ \ \ \text{for} \ \ \ \ x,y \in \{0,...,n-1\}$$
\begin{notation}\label{mult}
Let us denote this semi-simple decomposition by:
$$V(0,\alpha)^{\otimes n}=\bigoplus_{x,y\in \N\times \N}  T_n(x,y)\otimes V(x,n\alpha+y) $$ where $T_n(x,y)$ is the multiplicity space corresponding to the weight $(x,n\alpha+y)$. 
\end{notation}
We will codify this decomposition by a graph in the plane with integer coordinates, where each point will have a certain "weight".
\begin{definition}
We say that $D(n)$ is a diagram for $V(0,\alpha)^{\otimes n}$ if it is included in the lattice with integer coordinates and each point $(x,y)\in D(n)$ has associated a certain natural number, called weight, which is the following multiplicity: $$t_n(x,y)=dim \ T_n(x,y).$$ 

This encodes in the position $(x,y)$ from the diagram $D(n)$, the multiplicity of the highest weight module $V(x,n \alpha+y)$.
For the simplicity of the notation, we will think that the origin of the diagram $D(n)$ has the coordinates $(0,n \alpha)$. In this way, on the position $(x,y)$ from the diagram, we encode the multiplicity of the module whose fundamental weight is $(x, n\alpha+y)$, which means that it moves away from the origin with $x$ steps from $0$ and $y$ steps from $n\alpha$.\\
\end{definition}
As we can see, by reading the non-zero multiplicities associated to the points from the diagram $D(n)$, we can deduce the tensor decomposition of $V(0,\alpha)^{\otimes n}$ . Let us see some examples:
\begin{center}

$ Case \ \ \bold{ n=2 }$
$$V(0,\alpha)\otimes V(0,\alpha)=V(0,2\alpha)\oplus V(0,2\alpha+1)\oplus V(1,2\alpha). 
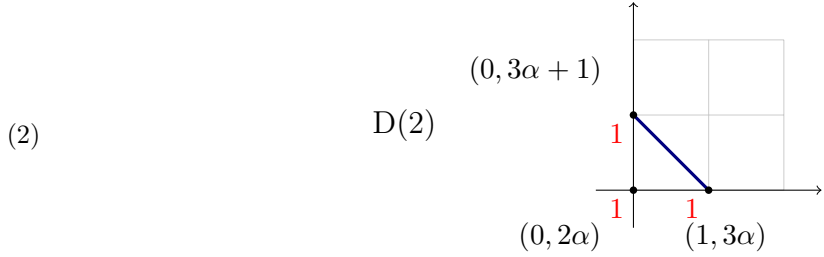
$$

\begin{equation}
\begin{split}
\begin{tikzpicture}
[x=1mm,y=1mm,scale=1,font=\large]
\node at (-25,12) [anchor=north east] { $\text{\Large D(2)}$ };
\draw[step=10mm,black!20,very thin] (0,0) grid (20,20);
\draw[very thick,black!50!blue] (10,0) -- +(-10,10);
\draw[->] (0,-5) -- (0,25);
\draw[->] (-5,0) -- (25,0);
\foreach \x/\y in {0/0, 10/0, 0/10} {\node at (\x,\y) [circle,fill,inner sep=1pt] {};}
\node at (0,0) [anchor=north east] {$\color{red}{1}$};
\node at (10,0) [anchor=north east] {$\color{red}1$};
\node at (0,10) [anchor=north east] {$\color{red}1$};
\node at (-3,-3) [anchor=north east] {$\color{black} \left( 0,2 \alpha \right)$};
\node at (19,-3) [anchor=north east] {$\color{black} \left( 1,3 \alpha \right)$};
\node at (-3,19) [anchor=north east] {$\color{black} \left( 0,3 \alpha+1 \right)$};
\end{tikzpicture}
\end{split}
\end{equation}
\end{center}
\begin{center}
$Case \ \ \bold{ n=3} $
$$V(0,\alpha)^{\otimes3}=\left( V(0,2\alpha)\oplus V(0,2\alpha+1)\oplus V(1,2\alpha)\right)\otimes V(0,\alpha)=$$
$$=\left( V(0,2\alpha)\otimes V(0,\alpha) \right) \oplus \left( V(0,2\alpha+1)\otimes V(0,\alpha)\right) \oplus \left( V(1,2\alpha)\otimes V(0,\alpha) \right)=$$
$$=\left( V(0,3\alpha) \oplus V(0,3\alpha+1) \oplus V(1,3\alpha) \right) \oplus$$
$$ \oplus  \left( V(0,3\alpha+1) \oplus V(0,3\alpha+2) \oplus V(1,3\alpha+1) \right) \oplus$$
$$ \oplus \left( V(1,3\alpha) \oplus V(1,3\alpha+1) \oplus V(0,3\alpha+1) \oplus V(2,3\alpha) \right).$$
We conclude that the semi-simple decomposition of the third tensor power is the following:
$$V(0,\alpha)^{\otimes3}= V(0,3\alpha) \oplus 3\cdot V(0,3\alpha+1) \oplus V(0,3\alpha+2) $$
$$ \oplus 2\cdot V(1,3\alpha)\oplus 2\cdot V(1,3\alpha+1)\oplus V(2,3\alpha).$$
\end{center}
We obtain the diagram for $D(3)$ in the following way:
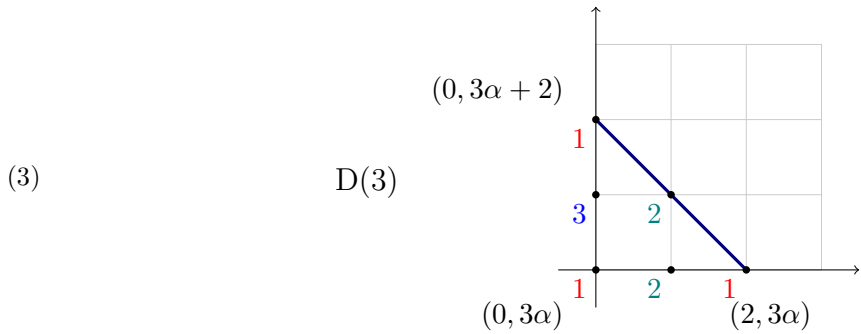
\begin{center}
\begin{equation}
\begin{split}
\begin{tikzpicture}
[x=1mm,y=1mm,scale=1,font=\large]
\node at (-25,15) [anchor=north east] { $\text{\Large D(3)}$ };
\draw[step=10mm,black!20,very thin] (0,0) grid (30,30);
\draw[very thick,black!50!blue] (20,0) -- +(-20,20);
\draw[->] (0,-5) -- (0,35);
\draw[->] (-5,0) -- (35,0);
\foreach \x/\y in {0/0, 10/0, 20/0, 0/10, 10/10, 0/20} {\node at (\x,\y) [circle,fill,inner sep=1pt] {};}
\node at (0,0) [anchor=north east] {$\color{red}1$};
\node at (10,0) [anchor=north east] {$\color{teal}2$};
\node at (20,0) [anchor=north east] {$\color{red}1$};
\node at (0,10) [anchor=north east] {$\color{blue}3$};
\node at (10,10) [anchor=north east] {$\color{teal}2$};
\node at (0,20) [anchor=north east] {$\color{red}1$};
\node at (-3,-3) [anchor=north east] {$\color{black} \left( 0,3 \alpha \right)$};
\node at (30,-3) [anchor=north east] {$\color{black} \left( 2,3 \alpha \right)$};
\node at (-3,27) [anchor=north east] {$\color{black} \left( 0,3 \alpha+2 \right)$};
\end{tikzpicture}
\end{split}
\end{equation}
\end{center}
\begin{center}
$$ Case \ \ \bold {n=4}$$
$$V(0,\alpha)^{\otimes4}=V(0,\alpha)^{\otimes3} \otimes V(0,\alpha)=$$
$$= \left( (V(0,3\alpha)\otimes V(0,\alpha) \right) \oplus 3\cdot \left( V(0,3\alpha+1) \otimes V(0,\alpha) \right) \oplus
\left( V(0,3\alpha+2) \otimes V(0,\alpha) \right) \oplus$$
$$ \oplus 2\cdot \left( V(1,3\alpha) \otimes V(0,\alpha) \right) \oplus 2\cdot \left( V(1,3\alpha+1) \otimes V(0,\alpha) \right) \oplus \left( V(2,3\alpha) \otimes V(0,\alpha) \right)=$$
$$=  \left( V(0,4\alpha) \oplus V(0,4\alpha+1) \oplus V(1,4\alpha) \right) \oplus$$
$$\oplus \left( 3V(0,4\alpha+1) \oplus 3V(0,4\alpha+2) \oplus 3V(1,4\alpha+1) \right) \oplus$$
$$ \oplus \left( V(0,4\alpha+2) \oplus V(0,4\alpha+3) \oplus V(1,4\alpha+2) \right) \oplus$$
$$ \oplus \left( 2V(1,4\alpha) \oplus 2V(1,4\alpha+1) \oplus 2V(0,4\alpha+1)\oplus 2V(2,4\alpha) \right) \oplus$$
$$ \oplus \left( 2V(1,4\alpha+1) \oplus 2V(1,4\alpha+2) \oplus 2V(0,4\alpha+2)\oplus 2V(2,4\alpha+1) \right) \oplus $$
$$ \oplus \left( V(2,4\alpha) \oplus V(2,4\alpha+1) \oplus V(1,4\alpha+1)\oplus V(3,4\alpha) \right).$$
We conclude that the semi-simple decomposition of the fourth tensor power is the following:
$$V(0,\alpha)^{\otimes4}=V(0,4\alpha) \oplus 6\cdot V(0,4\alpha+1) \oplus 6\cdot V(0,4\alpha+2) \oplus$$ 
$$\oplus V(0, 4\alpha+3) \oplus 3\cdot V(1,4\alpha) \oplus 8\cdot V(1,4\alpha+1) \oplus$$
$$\oplus 3\cdot V(1,4\alpha+2)\oplus 3\cdot V(2,4\alpha)\oplus 3\cdot V(2,4\alpha+1)\oplus V(3,4\alpha).$$
\end{center}
We obtain diagram $D(4)$ with the following weights:\\ 
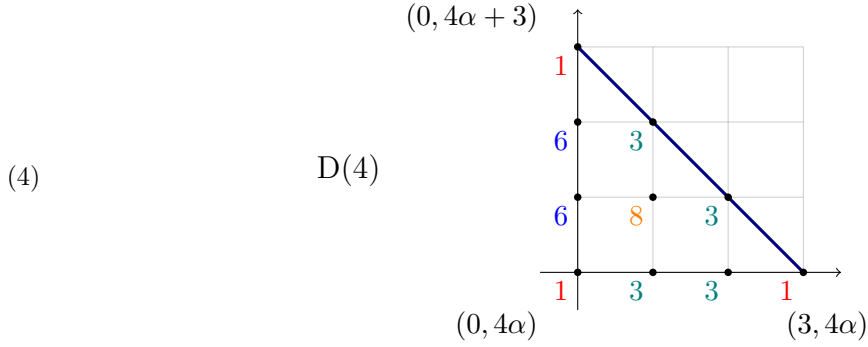
\begin{center}
\begin{equation}
\begin{split}
\begin{tikzpicture}
[x=1mm,y=1mm,scale=1,font=\large]
\node at (-25,17) [anchor=north east] { $ \text{\Large{D(4)}}$ };
\draw[step=10mm,black!20,very thin] (0,0) grid (30,30);
\draw[very thick,black!50!blue] (30,0) -- +(-30,30);
\draw[->] (0,-5) -- (0,35);
\draw[->] (-5,0) -- (35,0);
\foreach \x/\y in {0/0, 10/0, 20/0, 30/0, 0/10, 10/10, 20/10, 0/20, 10/20, 0/30} {\node at (\x,\y) [circle,fill,inner sep=1pt] {};}
\node at (0,0) [anchor=north east] {$\color{red}1$};
\node at (10,0) [anchor=north east] {$\color{teal}3$};
\node at (20,0) [anchor=north east] {$\color{teal}3$};
\node at (30,0) [anchor=north east] {$\color{red}1$};
\node at (0,10) [anchor=north east] {$\color{blue}6$};
\node at (10,10) [anchor=north east] {$\color{orange}8$};
\node at (20,10) [anchor=north east] {$\color{teal}3$};
\node at (0,20) [anchor=north east] {$\color{blue}6$};
\node at (10,20) [anchor=north east] {$\color{teal}3$};
\node at (0,30) [anchor=north east] {$\color{red}1$};
\node at (-4,-4) [anchor=north east] {$\color{black} \left( 0,4 \alpha \right)$};
\node at (40,-4) [anchor=north east] {$\color{black} \left( 3,4 \alpha \right)$};
\node at (-4,37) [anchor=north east] {$\color{black} \left( 0,4 \alpha+3 \right)$};
\end{tikzpicture}
\end{split}
\end{equation}
\end{center}

In the sequel, we will describe how the sequence of diagrams $D(n)$, can be constructed following an inductive procedure. More precisely, we will suppose that we know the diagram $D(n)$. Then, we will define a set of precise local moves such that if we apply these moves to each point from the diagram $D(n)$, we will be able to reconstruct the diagram $D(n+1).$

Let us start with a module $V(x,n\alpha+y)$ for $x,y \in \{0,...,n-1\}$. We will encode diagramatically the effect of tensoring this module with a standard module $V(0, \alpha)$. More specifically, as before, we will encode the simple modules that occur in the decomposition of $V(x,n\alpha+y)\otimes V(0,\alpha)$ in a weighted lattice. Let us think that initially, $V(x,n\alpha+y)$ is encoded by diagram $D$ which has just one point at the origin $(0,n\alpha)$ and the corresponding multiplicity $1$.\\ 

\begin{definition}
${\bold a)}$ From the $Theorem$ \ref{Thm}, for any $x \in \N\setminus\{0\}$ and any $y \in \N, n\in \N$ we have the decomposition:
$$V(0,\alpha)\otimes V(x,n \alpha+y)=V(x,(n+1)\alpha+y) \oplus V(x+1,(n+1)\alpha+y) \oplus $$ 
$$ \oplus V(x-1,(n+1)\alpha+y+1)\oplus V(x,(n+1)\alpha+y+1).$$
We call the effect of tensoring $V(x,n \alpha+y)$ with $V(0,\alpha)$ a ``blow up of type $(x,y)$'' and $B(x,y)$ the new corresponding diagram:

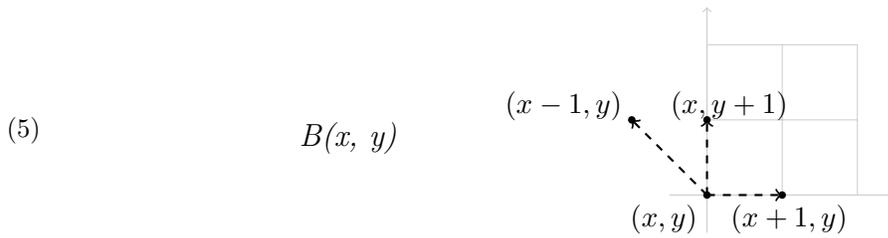
\begin{center}
\begin{equation}\label{eq:bla2}
\begin{split}
\begin{tikzpicture}
[x=1mm,y=1mm,scale=1,font=\large]
\draw[step=10mm,black!20,very thin] (0,0) grid (20,20);
\draw[->,black!20] (0,-5) -- (0,25);
\draw[->,black!20] (-5,0) -- (25,0);
\foreach \x/\y in {0/0, 10/0, 0/10, -10/10} {\node at (\x,\y) [circle,fill,inner sep=1pt] {};}
\node at (-40,11) [anchor=north east] { $\text{\Large{B(x, y)}}$ };
\node at (0,0) [anchor=north east] {$\left( x, y \right)$};
\node at (-10,15) [anchor=north east] {$\left( x-1, y \right)$};
\node at (12,15) [anchor=north east] {$\left( x, y+1 \right)$};
\node at (20,0) [anchor=north east] {$\left( x+1, y \right)$};
\draw[->,dashed,thick] (0,0) -- (0,10);
\draw[->,dashed,thick] (0,0) -- (10,0);
\draw[->,dashed,thick] (0,0) -- (-10,10);
\end{tikzpicture}
\end{split}
\end{equation}
\end{center}
$\bold {b)}$ If the first coordinate $x=0$, then for any $y \in \N$ and $n \in \N$, the decomposition is as follows:
$$V(0,\alpha)\otimes V(0,n \alpha+y)=V(0,(n+1)\alpha+y)\oplus V(1,(n+1)\alpha+y)\oplus V(0,(n+1)\alpha+y+1).$$
We call the effect of tensoring $V(0,n \alpha+x)$ with $V(0,\alpha)$ a ``blow up of type $(0,y)$'' and $B(0,y)$ the new corresponding diagram. 
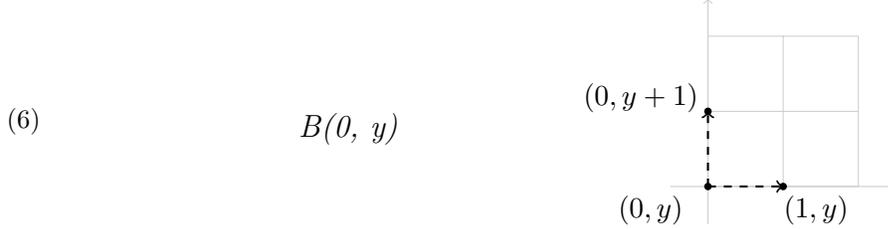
\begin{center}
\begin{equation}\label{eq:bla2}
\begin{split}
\begin{tikzpicture}
[x=1mm,y=1mm,scale=1,font=\large]
\draw[step=10mm,black!20,very thin] (0,0) grid (20,20);
\draw[->,black!20] (0,-5) -- (0,25);
\draw[->,black!20] (-5,0) -- (25,0);
\foreach \x/\y in {0/0, 10/0, 0/10} {\node at (\x,\y) [circle,fill,inner sep=1pt] {};}
\node at (-40,11) [anchor=north east] { $\text{\Large{B(0, y)}}$ };
\node at (0,15) [anchor=north east] {$\left( 0, y+1 \right)$};
\node at (20,0) [anchor=north east] {$\left( 1, y \right)$};
\node at (-2,0) [anchor=north east] {$\left(0, y\right)$};
\draw[->,dashed,thick] (0,0) -- (0,10);
\draw[->,dashed,thick] (0,0) -- (10,0);
%
\end{tikzpicture}
\end{split}
\end{equation}
\end{center}
\end{definition}
\begin{lemma}
The diagram $D(n+1)$ can be obtained from $D(n)$, by blowing up each point $(x,y)\in D(n)$ with $B(x,y)$ for $t_n(x,y)$ times and add in each vertex all the new multiplicities.
\end{lemma}
\begin{proof}
Suppose we have $D(n)$. This means that:
$$V(0,\alpha)^{\otimes n}=\bigoplus_{x,y\in \N\times \N} \ \left(t_n(x,y)\cdot V(x,n\alpha+y) \right) $$ 
In order to obtain the multiplicities that occur in $D(n+1)$, we have:\\
$$V(0,\alpha)^{\otimes {n+1}}=\bigoplus_{x,y\in \N\times \N} \left( t_n(x,y)\cdot \left(V(x,n\alpha+y)\otimes V(0,\alpha)\right) \right) \ (\ast)$$
On the other hand, the multiplicities $t_{n+1}$ occur in the following way:\\
$$V(0,\alpha)^{\otimes {n+1}}=\bigoplus_{x,y\in \N\times \N}\left( t_{n+1}(x,y)\cdot V(x,(n+1)\alpha+y) \right) $$
Using the previous formula $(\ast )$, we notice that the diagram $D(n+1)$, has some extra weights with respect to the diagram $D(n)$. More precisely, each term $(V(x,n\alpha+y)\otimes V(0,\alpha))$ will add to the weights from $D(n)$, some extra multiplicities corresponding to a blow up of center
$$\left(x+0,(n\alpha+y)+\alpha \right)= \left(x,(n+1)\alpha+y \right).$$
This is encoded in $D(n+1)$ as the blow-up $B(x,y)$ with center $(x,y)$. Counting the multiplicities, for each point $(x,y)$, we'll have to do the blow-up $B(x,y)$ for $t_n(x,y)$ times. In this way, we obtain $t_{n+1}(x,y)$.
\end{proof}
Up to this point, we saw how to construct the recursive relation that relates $D(n)$ and $D(n+1)$. However, this is still at the theoretical level. In the following part, we will use the fact that we know the initial step corresponding to the diagram $D(1)$, and apply the previous recursive relation. Finally, we will conclude that each multiplicity space $T_n(x,y)$ can be described in a natural way, using a method which constructs a set of paths in the plane, which satisfy certain restrictions.\\
\begin{remark}
Let us fix a point in the plane, with coordinates $(x,y)$. Then, the weight of $(x,y)$ in the diagram $D(n+1)$, $t_{n+1}(x,y)$ can be obtained by adding all the multiplicities of the points $(x',y')$ from $D(n)$, which can arrive to the point $(x,y)$ using one of the following moves:
$$M1) \ \ \ \ \ \ \ \text{stay  move} \ \ \ \ \ (x',y')\rightarrow (x',y') \ \ \ \ \ $$
$$M2) \ \ \ \ \ \  \ \longrightarrow  \ \ \ \ \ \ \ \ \ \ \ \ (x',y')\rightarrow (x'+1,y')$$
$$M3) \ \ \ \ \ \ \ \ \ \uparrow  \ \ \ \ \ \ \ \ \ \ \  \ \ (x',y')\rightarrow (x',y'+1) $$
$$ \ \ \ \ \ \ \ \ \ \ \ \ \ \ \ M4)  \ \ \ \ \ \ \ \ \nwarrow  \ \ \ \ \ \ \ \ \ \ \ \ (x',y')\rightarrow (x'-1,y'+1) \text{ if } x'>0$$
\end{remark}

Here, the reason for having the condition that the move $M_4$ can be done just if $x>0$, is the fact that a coefficient that decreases the $y$ coordinate occurs in the blow-up $B(x,y)$ if and only if $x>0$.
\begin{remark}
1) If we start from $D(n-1)$, we can obtain $D(n+1)$, by counting certain paths of length $2$ in the integer lattice (with the corresponding multiplicities as in $D(n-1)$.
In this way, applying twice the first remark we conclude that: 
$$t_{n+1}(x,y)= \text{card} \ \{ \  \text {paths of length 2 starting from points in } D(n-1) \ \text {which} $$
$$ \ \ \ \ \ \ \ \ \ \ \ \ \ \ \ \ \ \ \ \ \ \ \ \ \ \ \ \ \ \text{end in } (x,y), \text{ with the possible moves} \ M_1, M_2 , M_3 \text{ or} \ M_4 \ \} .$$

2) We will iterate this argument by induction, using as initial data the diagram $D(1)$ which has the following form:
\begin{equation}
\begin{split}
\begin{tikzpicture}
[x=1mm,y=1mm,scale=1/3,font=\large]
\draw[step=10mm,black!20,very thin] (0,0) grid (20,20);
\draw[->] (0,-5) -- (0,25);
\draw[->] (-5,0) -- (25,0);
\foreach \x/\y in {0/0} {\node at (\x,\y) [circle,fill,inner sep=1pt] {};}
\node at (0,0) [anchor=north east] {$\color{red}1$};
\node at (-60,15) [anchor=north east] { $D(1)$ };
\end{tikzpicture}
\end{split}
\end{equation}
We obtain the following combinatorial description for the multiplicity spaces:
\end{remark}
\begin{theorem}\label{multiplicities}
In $D(n+1)$, for each point $(x,y)\in \Z\times \Z$ the associated multiplicity has the formula:
$$t_{n+1}(x,y)= \text {number of paths in the plane from} \ (0,0) \ \text {to} \ (x,y) \text{ of length } \ (n) $$
$$ \ \ \ \ \ \ \ \  \text {and possible moves M1, M2, M3 or M4 with the condition that} $$
$$ \ \ \ \ \ \ \ \ \ \ \ \text{ they do not have any point with a negative coordinate on the x-axis.}$$
\end{theorem}
Coming back to the semi-simple decomposition of the tensor product, we obtain the correspondence given in Theorem \ref{Thm1}.
$$ T_{n+1}(x,y) \longleftrightarrow \{ \sigma  \subseteq \R^2 \mid \sigma \text{ is a path of length n, in the first quadrant }  $$
$$ \ \ \ \ \ \ \ \ \ \ \ \ \ \ \ \  \text{ between the points } (0,0) \text{ to } (x,y) $$
$$ \text{ \ \ \ \  \ \ \ \ \ \ \ \ \ \ \ \ \ \ \ \ \ \ \ \ \ \ \ \ \ \ \ \ \ constructed with the moves } M_1, M_2, M_3 \text{ or} \  M_4 \} $$\label{first}

\section{Proof of the Conjecture}{\label{sec6}}

This section is the second part of the proof of the Wagner-Marin Conjecture (Theorem \ref{conj}) and concerns a purely combinatorial result, which deals with the identification between two counting problems. Following the previous section (\ref{sec5}), we conclude that the cardinality of the multiplicity spaces corresponding to $V(0,\alpha)^{\otimes n+1}$ is the same as the number of paths in the plane of fixed length $n$, which follow some fixed rules. On the other hand, it was known that the conjectured dimension can be interpreted as the number of pairs of planar paths with certain requirements. We will show that the two counting problems coincide, by establishing a bijection between the previous set of paths towards the set of pairs of paths with the required properties.

\begin{remark}
We will denote by $\Delta _{n} \in \R^2$ the standard $2$-dimensional simplex, which has the length of the edge $n$. We notice that in the diagram $D(n+1)$, the points with non-zero weights are inside (or on the boundary) of $\Delta_{n}$.
\end{remark}
\begin{notation}

Consider the following set of paths
$$P_{n+1}(x,y):= \{ \text{planar paths from} \ (0,0) \text{ to} \ (x,y) \ 
\text {of length} \  n  $$
$$  \ \ \ \ \ \ \ \ \ \ \ \ \ \ \ \ \ \text { constructed using the possible moves} $$
$$  \ \ \ \ \ \text {M1, M2, M3 or M4, }$$
$$  \ \ \ \ \ \ \ \ \ \ \ \ \ \ \ \text{ which are contained in the positive quadrant} \} $$
\end{notation}
\begin{remark}\label{set}
As we have seen in \ref{multiplicities}, for any number $n \in \N$ and $x,y \in \Delta _{n}$, we have the following bijection:
$$\psi^{n+1}_{(x,y)}:T_{n+1}(x,y) \rightarrow \ P_{n+1}(x,y)$$\label{ggg}
$$t_{n+1}(x,y)=| \ P_{n+1}(x,y) |$$
\end{remark}

\subsection{Intertwinner spaces}\label{int}
We will begin by establishing the relation between the centraliser algebra $LG_{n+1}(\alpha)$ and the intertwinner spaces corresponding to the tensor powers of simple components that occur in the semi-simple decomposition of $V(0,\alpha)^{\otimes (n+1)}$.
\begin{remark} As we have seen, there is the following tensor decomposition:
$$V(0,\alpha)^{\otimes n+1}=\bigoplus_{x,y \in \Delta_{n}}\Big( T_{n+1}(x,y) \otimes V\big(x,(n+1)\alpha+y \big) \ \Big)$$ 
where $T_{n+1}(x,y)$ is the multiplicity space corresponding to the weight $(x,(n+1) \alpha+y)$.
\end{remark}
\begin{proposition}\label{end}
From \cite{GP}, for typical $(n,\alpha) \in \Lambda$, $V(n,\alpha)$ is simple representation and moreover:
$$Hom_{U_q(sl(2|1))}(V(n,\alpha),V(m,\beta))\simeq\delta _{(n,\alpha)}^{(m,\beta)} \cdot \Bbbk Id.$$ 
\end{proposition}
\begin{definition}
For a representation $V \in Rep(U_q(sl(2|1)))$ and a natural number  $n$, the $n^{th}$ intertwinner space corresponding to $V$ is the following endomorphism ring:
$$\mathscr{I}_{n}(V):=End_{U_q(sl(2|1))}\left( \bigoplus_{i=1}^{n}  V_i \ \right), \text{  where  }  V_i \simeq V \ \  \forall i \in \{1,...,n\} $$
\end{definition}
In our case, using that $\alpha \notin \Q$, we deduce that for any $(x,y) \in \Delta _{n+1}$, all modules $V(x,(n+1)\alpha+y)$ are typical. 
\begin{remark}From \ref{end}, we see that the intertwinner spaces for simple representations are actually rings of matrices:  
$$\mathscr{I}_{n+1}(V(0,\alpha))\simeq M({n+1},\Bbbk)$$ 
\end{remark}
The previous remarks lead to the following characterisation the centraliser algebra $LG_{n+1}(\alpha)$, in terms of the intertwinner spaces corresponding to the multiplicity spaces of the simple components from $V(0,\alpha)^{n+1}$:
\begin{proposition}\label{inter}
$$LG_{n+1}(\alpha)\simeq \bigoplus_{x,y \in \Delta_{n}} \mathscr{I}_{t_{n+1}(x,y)}\Big(V\big(x,(n+1)\alpha+y\big)\Big)\simeq \bigoplus_{x,y \in \Delta_{n}} M(t_{n+1}(x,y),\Bbbk).$$
\end{proposition}
\begin{proof}
$$LG_{n+1}(\alpha)\simeq End_{U_q(sl(2|1))}\left( V(0,\alpha)^{\otimes (n+1)}\right) \simeq$$
$$ \simeq \ End_{U_q(sl(2|1))} \left( \bigoplus_{x,y \in \Delta_{n}}  \left( T_{n+1}(x,y) \otimes V(x,(n+1)\alpha+y)) \right) \right) \simeq $$
$$ \simeq^{\text{Prop}\ref{end}} \bigoplus_{x,y \in \Delta_{n}} \ End_{U_q(sl(2|1))} \Big( T_{n+1}(x,y) \otimes V(x,(n+1)\alpha+y)) \Big) \simeq $$

$$ \simeq \bigoplus_{x,y \in \Delta_{n}} \mathscr{I}_{t_{n+1}}(x,y)(V(0,\alpha))\Big(V\big(x,(n+1)\alpha+y\big)\Big)$$
\end{proof}
\begin{corollary}
This decomposition in terms of matrix rings, leads to the following formula for the dimension of the centraliser algebra:
$$dim \ LG_{n+1}(\alpha)=\sum_{x,y\in \Delta_{n}} t_{n+1}(x,y)^2.$$
\end{corollary}
\begin{corollary}\label{1}
Using the identification from Proposition \ref{set},  we conclude that:
$$dim \ LG_{n+1}(\alpha)=\sum_{x,y\in \Delta_{n}}|P_{n+1}(x,y)|^2.$$
\end{corollary}
\subsection{Conjectured dimension described using pairs of planar paths}

In \cite{MW}, it is mentioned that F. Chapoton remarked that the conjectured dimension of $LG_{n+1}$ coincides with a combinatorial formula for a way of counting pairs of paths in the plane.
\begin{notation} Consider the following set of pairs of planar paths:
 $$C_{n+1}=\{ (\sigma, \tau) \subseteq \R^2 \mid \sigma, \tau \text{are disjoint paths in the } (n+1)\times (n+1) \text { square} $$ 
$$\text { \ \ \ \ \ \ \ \ \ \ \ \ \ \ \ \ \ constructed using two possible moves- upwards} \uparrow  \text{or right} \rightarrow $$
$$ \ \ \ \ \ \ \ \ \ \ \ \ \ \ \ \ \ \text {between the points } (0,1)\rightarrow (n,n+1) \text { and } (1,0)\rightarrow (n+1,n) \} $$
\end{notation}

\begin{theorem}(\cite{A}, \cite{S}) \label{2}(Counting pairs of planar paths)
The following combinatorial identity holds:
$$ | C_{n+1} |=\frac{(2n)!(2n+1)!}{(n!(n+1)!)^2}$$
\end{theorem}
\subsection{Relation between pairs of paths in the square and pairs of paths in the simplex $\Delta_{n+1}$}
Putting together the results from the last sections, we notice that there are combinatorial models for both sides which appear in relation \ref{conj}. In this part, we will concentrate on the set $C_{n+1}$, and describe it as an union of certain subsets of pairs of paths.
\begin{remark}
The description from \ref{first} for the multiplicity spaces and the relation \ref{1}, provide a combinatorial description for the centraliser algebra $LG_{n+1}(\alpha)$. On the other hand, we will use the result that was known from Theorem \ref{2}. From these results, the question that arises now is to establish a relation between the sets 
$$ \bigcup_{x,y \in \Delta_n} \left( P_{n+1}(x,y) \times P_{n+1}(x,y) \right) \ \ \ \ \longleftrightarrow^{?} \ \ \ \ C_{n+1}. $$\label{?}
\end{remark}

For the moment, in the previous formula \ref{?}, there is an asymmetry coming from the fact that in the left hand side one has a union of sets indexed by points with integer coordinates from the simplex $\Delta_n$, whereas in the right hand side there is a global set, which is not yet partitioned in smaller subsets. This rises a natural question, to understand the set of pairs of paths $C_{n+1}$, as an union of subsets, indexed by the same indexing set corresponding to $\Delta_n$. In order to do this, the main idea in the next part is to start with a pair of paths in the square, and remember where those paths "cut the principal diagonal". We will use this data as an indexing set.
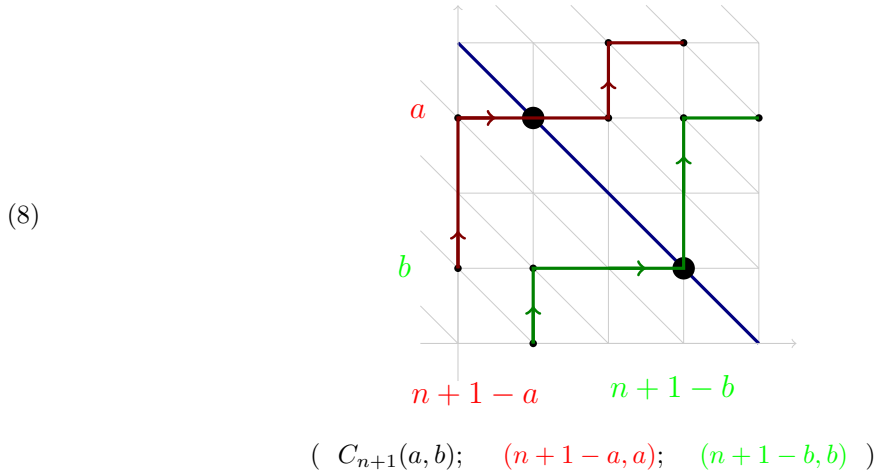
\begin{center}
\begin{equation}\label{eq:bla4}
\begin{split}
\begin{tikzpicture}
[x=1mm,y=1mm,scale=1,font=\Large]
\node at (12,-4) [anchor=north east] {$\Large{\color{red}n+1-a}$};
\node at (-3,33) [anchor=north east] {$\Large{\color{red}a}$};
\node at (38,-3) [anchor=north east] {$\Large{\color{green}n+1-b}$};
\node at (-5,13) [anchor=north east] {$\Large{\color{green}b}$};
\draw[step=10mm,black!20,very thin] (0,0) grid (40,40);
\draw[black!20,very thin] (0,0) -- +(-5,5);
\draw[black!20,very thin] (10,0) -- +(-15,15);
\draw[black!20,very thin] (20,0) -- +(-25,25);
\draw[black!20,very thin] (30,0) -- +(-35,35);
\draw[very thick,black!50!blue] (40,0) -- +(-40,40);
\draw[black!20,very thin] (40,10) -- +(-35,35);
\draw[black!20,very thin] (40,20) -- +(-25,25);
\draw[black!20,very thin] (40,30) -- +(-15,15);
\draw[black!20,very thin] (40,40) -- +(-5,5);
\draw[->,black!20] (0,-5) -- (0,45);
\draw[->,black!20] (-5,0) -- (45,0);
\foreach \x/\y in {0/10, 0/30, 20/30, 20/40, 30/40, 10/0, 10/10, 30/10, 30/30, 40/30} {\node at (\x,\y) [circle,fill,inner sep=1pt] {};}
\foreach \x/\y in {10/30, 30/10} {\node at (\x,\y) [circle,fill,inner sep=3pt]{};}
\draw[very thick,black!50!green] (10,0) -- (10,10) -- (30,10) -- (30,30) -- (40,30);
\draw[very thick,black!50!green,->] (10,0) -- (10,5);
\draw[very thick,black!50!green,->] (20,10) -- (25,10);
\draw[very thick,black!50!green,->] (30,20) -- (30,25);
\draw[very thick,black!50!red] (0,10) -- (0,30) -- (20,30) -- (20,40) -- (30,40);
\draw[very thick,black!50!red,->] (0,10) -- (0,15);
\draw[very thick,black!50!red,->] (0,30) -- (5,30);
\draw[very thick,black!50!red,->] (20,30) -- (20,35);
\end{tikzpicture}
\end{split}
\end{equation}
$$ \left( \ \  C_{n+1}(a,b); \ \ \ \ \color{red} (n+1-a,a)\color{black}{;} \ \ \ \ \color{green}(n+1-b,b) \ \  \color{black}\right)$$ 
\end{center}
\begin{definition}
Following this procedure, we will consider the set of paths that cut the diagonal in two given points.
Let $(a,b)\in \N\times \N$ with $a,b\leq n+1$ and $a>b$. Let us denote by:
$$C_{n+1}(a,b):= \{ \text {pairs of paths in } C_{n+1} \text{ that cut the principal diagonal of the} $$
$$ \ \ \ \ \ \ \ \ \ \ \ \ \ \ \ \ \ \ \ \text{square precisely in the points} \  (n+1-a,a) \ \text {and} \ (n+1-b,b) \} .$$
\end{definition}
\begin{notation}
Let us denote by 
$$\bar{\Delta}_{n}:=\{ (a',b') \in \Z \times \Z \mid a',b' \leq n, a \geq b  \}$$
$$\bar{\Delta}^{sym}_{n}:=\{ (a',b') \in \Z \times \Z \mid a',b' \leq n+1, a > b  \}$$
\end{notation}

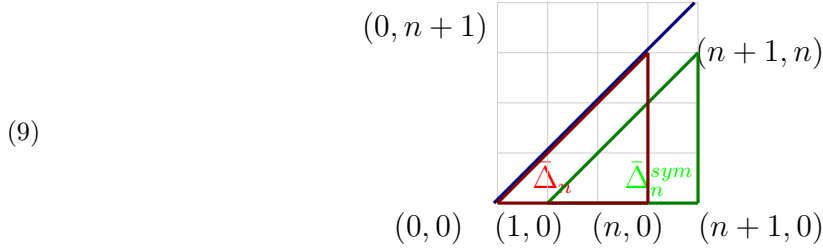
\begin{center}
\begin{equation}
\begin{split}
\begin{tikzpicture}
[x=1mm,y=1mm,scale=2/3,font=\Large]
\node at (-5,0) [anchor=north east] {$\color{black} (0,0)$};
\node at (67,0) [anchor=north east] {$\color{black} (n+1,0)$};
\node at (0,40) [anchor=north east] {$\color{black} (0,n+1)$};
\node at (15,0) [anchor=north east] {$\color{black} (1,0)$};
\node at (35,0) [anchor=north east] {$\color{black} (n,0)$};
\node at (67,35) [anchor=north east] {$\color{black} (n+1,n)$};
\node at (17,10) [anchor=north east] {$\Huge{\color{red} \bar{\Delta}_{n}}$};
\node at (41,10) [anchor=north east] {$\Huge{\color{green} \bar{\Delta}^{sym}_{n}}$};
\draw[step=10mm,black!20,very thin] (0,0) grid (40,40);
\draw[very thick,black!50!blue] (-0.7,0) -- +(40,40);
\draw[very thick,black!50!green] (10,0) -- +(30,30);
\draw[very thick,black!50!green] (10,0) -- +(30,0);
\draw[very thick,black!50!green] (40,0) -- +(0,30);
\draw[very thick,black!50!red] (0,0) -- +(30,30);
\draw[very thick,black!50!red] (0,0) -- +(30,0);
\draw[very thick,black!50!red] (30,0) -- +(0,30);
\end{tikzpicture}
\end{split}
\end{equation}
\end{center}
\begin{remark}\label{disjunion}
Using these notations, $C_{n+1}$ can be expressed in the following way:
$$ \  C_{n+1}=\bigcup_{(a,b)\in \bar{\Delta}^{sym}_n} C_{n+1}(a,b)$$
\end{remark}
A second asymmetry that one can notice in the question \ref{?}, is related to the fact that in the left hand side one has products of sets of planar paths, but in the right hand side there is a set whose elements are certain pairs of paths. The second idea of this part, will be to cut each pair of paths from from $C_{n+1}(a,b)$ using the secondary diagonal of the $(n+1)\times (n+1)$ square. Pursuing this procedure, one gets two pairs of paths in the simplex $\Delta_{n+1}$. Now we will make this precise.

\begin{definition}
For $(a,b)\in \bar{\Delta}^{sym}_n$, consider the following set:
$${C_{n+1}^{\Delta}}(a,b):=\{ \text{pairs of disjoint paths in the simplex } \Delta _{n+1}, \text{ constructed with moves } M_2 \text { or } M_3,$$
$$ \ \ \ \ \ \ \ \ \ \ \ \ \ \ \ \ \ \ \ \text{ between the points } (1,0) \rightarrow (n+1-a,a) \text{ and } (0,1) \rightarrow (n+1-b,b) \text{ respectively} \} .$$
\end{definition}
\begin{center}
\begin{equation}\label{eq:bla4}
\begin{split}
\begin{tikzpicture}
[x=1mm,y=1mm,scale=4/5,font=\Large]
\draw[step=10mm,black!20,very thin] (0,0) grid (40,40);
\draw[black!20,very thin] (0,0) -- +(-5,5);
\draw[black!20,very thin] (10,0) -- +(-15,15);
\draw[black!20,very thin] (20,0) -- +(-25,25);
\draw[black!20,very thin] (30,0) -- +(-35,35);
\draw[very thick,black!50!blue] (40,0) -- +(-40,40);
\draw[black!20,very thin] (40,10) -- +(-35,35);
\draw[black!20,very thin] (40,20) -- +(-25,25);
\draw[black!20,very thin] (40,30) -- +(-15,15);
\draw[black!20,very thin] (40,40) -- +(-5,5);
\draw[->,black!20] (0,-5) -- (0,45);
\draw[->,black!20] (-5,0) -- (45,0);
\node at (12,-4) [anchor=north east] {$\Large{\color{red}n+1-a}$};
\node at (-3,33) [anchor=north east] {$\Large{\color{red}a}$};
\node at (38,-3) [anchor=north east] {$\Large{\color{green}n+1-b}$};
\node at (-5,13) [anchor=north east] {$\Large{\color{green}b}$};
\foreach \x/\y in {0/10, 0/30, 10/0, 10/10, 30/10, 10/30} {\node at (\x,\y) [circle,fill,inner sep=1pt] {};}
\foreach \x/\y in {10/30, 30/10} {\node at (\x,\y) [circle,fill,inner sep=2pt]{};}
\draw[very thick,black!50!green] (10,0) -- (10,10) -- (30,10); 
\draw[very thick,black!50!green,->] (10,0) -- (10,5);
\draw[very thick,black!50!green,->] (20,10) -- (25,10);
\draw[very thick,black!50!red] (0,10) -- (0,30) -- (10,30);
\draw[very thick,black!50!red,->] (0,10) -- (0,15);
\draw[very thick,black!50!red,->] (0,30) -- (5,30);
\node at (92,-4) [anchor=north east] {$\Large{\color{red}n+1-a}$};
\node at (77,33) [anchor=north east] {$\Large{\color{red}a}$};
\node at (118,-3) [anchor=north east] {$\Large{\color{green}n+1-b}$};
\node at (75,13) [anchor=north east] {$\Large{\color{green}b}$};
\draw[step=10mm,black!20,very thin] (80,0) grid (120,40);
\draw[black!20,very thin] (80,0) -- +(-5,5);
\draw[black!20,very thin] (90,0) -- +(-15,15);
\draw[black!20,very thin] (100,0) -- +(-25,25);
\draw[black!20,very thin] (110,0) -- +(-35,35);
\draw[very thick,black!50!blue] (120,0) -- +(-40,40);
\draw[black!20,very thin] (120,10) -- +(-35,35);
\draw[black!20,very thin] (120,20) -- +(-25,25);
\draw[black!20,very thin] (120,30) -- +(-15,15);
\draw[black!20,very thin] (120,40) -- +(-5,5);
\draw[->,black!20] (80,-5) -- (80,45);
\draw[->,black!20] (75,0) -- (125,0);
\foreach \x/\y in {90/30, 100/30, 100/40, 110/40, 110/10, 110/30, 120/30} {\node at (\x,\y) [circle,fill,inner sep=1pt] {};}
\foreach \x/\y in {90/30, 110/10} {\node at (\x,\y) [circle,fill,inner sep=2pt]{};}
\draw[very thick,black!50!green] (110,10) -- (110,30) -- (120,30);
\draw[very thick,black!50!green,->] (110,10) -- (110,15);
\draw[very thick,black!50!green,->] (110,30) -- (115,30);
\draw[very thick,black!50!red] (90,30) -- (100,30) -- (100,40) -- (110,40);
\draw[very thick,black!50!red,->] (90,30) -- (95,30);
\draw[very thick,black!50!red,->] (100,30) -- (100,35);
\draw[very thick,black!50!red,->] (100,40) -- (105,40);
\end{tikzpicture}
\end{split}
\end{equation}
$$ \left( \ \  C_{n+1}^{\Delta}(a,b); \ \ \color{red} (n+1-a,a)\color{black}{;} \ \ \color{green}(n+1-b,b) \ \  \color{black}\right) \ \ \ \ \ \ \  \  \left( \ \  C_{n+1}^{\Delta}(a,b); \ \ \color{red} (n+1-a,a)\color{black}{;} \ \ \color{green}(n+1-b,b) \ \  \color{black}\right)$$  
\end{center}
\begin{lemma} 
There is the following one to one correspondence of sets:
$$ \ \ \ \ \ \ \ \ \ \ \ \ \ C_{n+1}(a,b)\simeq {C_{n+1}^{\Delta}} (a,b)\times {C_{n+1}^{\Delta}}(a,b), \ \ \ \forall \ (a,b)\in \bar{\Delta}^{sym}_n$$
\end{lemma}
This bijection can be established as we have seen, above.  If we start with a path from $C_{n+1}(a,b)$ and cut it along the secondary diagonal of the square, we get two paths in ${C_{n+1}^\Delta} (a,b)$. 

\begin{proposition}\label{splitpairs}
From the previous remarks and definitions we conclude that:
$$C_{n+1}=\bigcup_{(a,b)\in \bar{\Delta}^{sym}_n} \left( {C_{n+1}^\Delta }(a,b)\times{C_{n+1}^\Delta }(a,b) \ \right)$$ 
$$\mid C_{n+1}\mid=\sum_{(a,b)\in \bar{\Delta}^{sym}_n} \mid {C_{n+1}^\Delta }(a,b)\mid ^2$$
\end{proposition}
In the following part, we will modify a bit the elements of the set $C_{n+1}^{\Delta }(a,b)$ and consider a new set where we allow intersections of paths, but not one crossing the other. This is just a technical detail that it will help in the last part.
\begin{notation}
For $(a,b)\in \N\times \N$ with $a,b\leq n$ and $a\geq  b$, denote by:
$${D_{n}^\Delta} (a,b):= \{ \text{ pairs of paths in the simplex } \Delta _{n}, \text{constructed with two possible moves} \uparrow  \text{or} \rightarrow $$
$$ \ \ \ \ \ \ \ \ \ \ \ \ \ \ \text{ between the points } \left( (0,0); (n+1-a,a) \right) and \left( (0,0); (n+1-b,b) \right) respectively, $$
$$ \ \ \ \ \ \ \ \ \ \ \ \ \ \ \ \ \ \ \text{such that they can intersect each other, but they do not cross each other}\}.$$  
\end{notation}
\begin{proposition} \label{CD}There is a bijection between the following sets:
$$\eta^{n+1}_{(a,b)}:{C_{n+1}^\Delta }(a,b)\simeq{D_{n}^\Delta } (a-1,b), \ \ \ \forall (a,b)\in \bar{\Delta}^{sym}_n $$
\end{proposition}
\begin{proof}
Let $(C, D)\in {C_{n+1}^\Delta }(a,b)$ be a pair of paths in the simplex $\Delta_{n+1}$, as in the picture below.\\ 
We define a new path in the following way:
$$D':=D+(-1,1)\subseteq \Delta_{n+1}$$ 
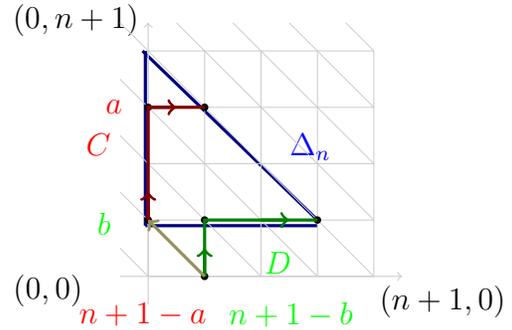
\begin{center}
\begin{equation}\label{eq:bla4}
\begin{split}
\begin{tikzpicture}
[x=1mm,y=1mm,scale=3/4,font=\Large]
\node at (-70,25) [anchor=north east] {$( \ \Large{\color{red}C}, \Large{\color{green}D} \ )\in {C_{n+1}^\Delta }(a,b)$};
\node at (12,-3) [anchor=north east] {$\Large{\color{red}n+1-a}$};
\node at (-3,33) [anchor=north east] {$\Large{\color{red}a}$};
\node at (-5,27) [anchor=north east] {$\Large{\color{red}C}$};
\node at (38,-3) [anchor=north east] {$\Large{\color{green}n+1-b}$};
\node at (27,6) [anchor=north east] {$\Large{\color{green}D}$};
\node at (-5,13) [anchor=north east] {$\Large{\color{green}b}$};
\node at (34,27) [anchor=north east] {$\Large{\color{blue} \Delta_{n}}$};
\node at (-10,2) [anchor=north east] {$\color{black} (0,0)$};
\node at (65,0) [anchor=north east] {$\color{black} (n+1,0)$};
\node at (0,50) [anchor=north east] {$\color{black} (0,n+1)$};
\draw[very thick,black!50!blue] (30,10) -- +(-31,30);
\draw[very thick,black!50!blue] (-0.5,9) -- +(30.5,0);
\draw[very thick,black!50!blue] (-0.5,9) -- +(0,31);
\draw[step=10mm,black!20,very thin] (0,0) grid (40,40);
\draw[black!20,very thin] (0,0) -- +(-5,5);
\draw[black!20,very thin] (10,0) -- +(-15,15);
\draw[black!20,very thin] (20,0) -- +(-25,25);
\draw[black!20,very thin] (30,0) -- +(-35,35);
\draw[black!20,very thin] (40,0) -- +(-45,45);
\draw[black!20,very thin] (40,10) -- +(-35,35);
\draw[black!20,very thin] (40,20) -- +(-25,25);
\draw[black!20,very thin] (40,30) -- +(-15,15);
\draw[black!20,very thin] (40,40) -- +(-5,5);
\draw[->,black!20] (0,-5) -- (0,45);
\draw[->,black!20] (-5,0) -- (45,0);
\foreach \x/\y in {0/10, 0/30, 10/0, 10/10, 30/10, 10/30} {\node at (\x,\y) [circle,fill,inner sep=1pt] {};}
\draw[very thick,black!50!green] (10,0) -- (10,10) -- (30,10); 
\draw[very thick,black!50!green,->] (10,0) -- (10,5);
\draw[very thick,black!50!green,->] (20,10) -- (25,10);
\draw[very thick,black!50!red] (0,10) -- (0,30) -- (10,30);
\draw[very thick,black!50!red,->] (0,10) -- (0,15);
\draw[very thick,black!50!red,->] (0,30) -- (5,30);
%
\draw[very thick,black!50!yellow,->] (10,0) -- (0,10);
\end{tikzpicture}
\end{split}
\end{equation}
\end{center}

\begin{center}
\begin{equation}\label{eq:bla5}
\begin{split}
\begin{tikzpicture}
[x=1mm,y=1mm,scale=3/4,font=\Large]
\node at (-35,15) [anchor=north east] {$\big( \ \Large{\color{red}C'}, \Large{\color{green}D'=D+(-1,1) } \ \big)\in {D_{n}^\Delta }(a-1,b)$};
\node at (12,-4) [anchor=north east] {$\Large{\color{red}n+1-a}$};
\node at (0,23) [anchor=north east] {$\Large{\color{red}a-1}$};
\node at (-4,18) [anchor=north east] {$\Large{\color{red}C'}$};
\node at (28,-4) [anchor=north east] {$\Large{\color{green}n-b}$};
\node at (18,7) [anchor=north east] {$\Large{\color{green}D'}$};
\node at (-2,13) [anchor=north east] {$\Large{\color{green}b}$};
\node at (26,20) [anchor=north east] {$\Large{\color{blue} \Delta_{n}}$};
\node at (-13,0) [anchor=north east] {$\color{black} (0,0)$};
\node at (45,0) [anchor=north east] {$\color{black} (n,0)$};
\node at (-2,35) [anchor=north east] {$\color{black} (0,n)$};
\draw[very thick,black!50!blue] (30,0) -- +(-30,30);
\draw[very thick,black!50!blue] (0,0) -- +(30,0);
\draw[very thick,black!50!blue] (0,0) -- +(0,30);
\draw[step=10mm,black!20,very thin] (0,0) -- (0,30) -- (30,0) -- cycle;
\draw[step=10mm,black!20,very thin] (0,10) -- (20,10) -- (20,0);
\draw[step=10mm,black!20,very thin] (0,20) -- (10,20) -- (10,0);
\draw[very thick,black!50!red] (0,0) -- (0,20) -- (10,20);
\draw[very thick,black!50!green] (0.5,0) -- (0.5,10) -- (20,10);
\end{tikzpicture}
\end{split}
\end{equation}
\end{center}

We will consider the simplex $\Delta _n$ inside $\Delta_{n+1}$, bounded by the points $(0,1),(n,1)$, and $(n+1,0)$. In this way, we will obtain a pair of paths $(C,D')\in {D_{n}^\Delta }(a-1,b)$. After that, it can be easily shown that this function is a bijection. 
\end{proof}
Putting together the correspondences from above, using Proposition \ref{splitpairs} and  Proposition \ref{CD} we conclude the following lemma. 
\begin{lemma}\label{C}
The cardinality of $C_{n+1}$ can be expressed using pairs of paths in the simplex $\Delta_n$:
$$\mid C_{n+1}\mid=\sum_{(a',b')\in \bar{\Delta}_n} \mid {D_{n}^\Delta }(a',b')\mid ^2.$$
\end{lemma}
\subsection{Correspondence between paths in the square and pairs of paths in the simplex}

\

In this section we will finish the proof of Theorem \ref{conj}. In this final part, we will show a correspondence between the sets $P_{n+1}(x,y)$ and ${D_{n}^\Delta }(a,b)$. 
We begin with the following observation.
\begin{remark}
For any pair of paths $C_1=((C_1)^k_x,(C_1)^k_y)$ and $C_2=((C_2)^k_x,(C_2)^k_y)\in {D_{n}^\Delta }(a,b)$, the condition that they do not cross each other can be stated as:
$$ \forall k , \ (C_1)^k_y\leq (C_2)^k_y.$$ 
\end{remark}
We will use this in the sequel and make a connection between the set $D^{\Delta}_n$, coming from the combinatorial model for the conjectured dimension and the set of paths $P_{n+1}$ which comes from the multiplicity spaces related to the centraliser algebra $LG_{n+1}$.
\begin{theorem} \label{3}
We have the following correspondence for any point $(x_0,y_0)\in \Delta _n$:
$$f^{n+1}_{(x_0,y_0)}:{D_{n}^\Delta }(x_0+y_0,y_0) \rightarrow P_{n+1}(x_0,y_0).$$
\end{theorem}
\begin{proof}
Let $C_1,C_2\in {D_{n}^\Delta }(x+y,y)$. This pair of paths can be encoded in a sequence of moves of four types. We start at the common point $(0,0)$ which correspond to the step $0$. Then, we see how each of the two paths changes from one step to the other. \\ 
Suppose that $(x_1,y_1)\in C_1$ and $(x_2,y_2)\in C_2$ which correspond to the  $k^{th}$ step. In order to pass to the $(k+1)^{st}$ step, we have four possibilities.

\

{\bf Movements corresponding to ${D_{n}^\Delta }(x+y,y)$}

\

The pair $\left( (x_1,y_1), \ (x_2,y_2)\right )$ for which we know the condition $y_2\geq y_1$, can be modified in order to arrive at the next step, by adding one out of the four following pairs of vectors:
            $$\left((0,1), \ (0,1)\right)$$
            $$\left((1,0), \ (1,0)\right)$$
            $$\left((1,0), \ (0,1)\right)$$
            $$\left((0,1), \ (1,0)\right)$$   
On the other hand, any path $C\in P_{n+1}(x,y)$, can be also encoded by the moves that occur from the $k^{th}$ step to the $(k+1)^{st}$step.

\

{\bf Movements corresponding to $P_{n+1}(x,y)$}:  

\

Let $(x,y)$ a point in that occur in the path $C$. We know the condition that $x\geq 0$. In order to pass to the next point in the path, we can modify the coordinates of the point $(x,y)$ with one out of the following vectors:
            $$M1) \longleftrightarrow(0,0)$$
            $$M3)\longleftrightarrow(0,1)$$
            $$M2)\longleftrightarrow(1,0)$$
            $$ \ \ M4)\longleftrightarrow(-1,1)$$
Now, we want to establish a correspondence between the two types of movements. We will define a function 
$$f^{n+1}_{(x_0,y_0)}:{D_{n}^\Delta }(x_0+y_0,y_0)\rightarrow P_{n+1}(x_0,y_0).$$\label{fff}
Let $C_1,C_2\in {D_{n}^\Delta }(x+y,y)$. We want to send each pair of points $\left( (x_1,y_1)\in C_1,(x_2,y_2)\in C_2\right )$ into $f^{n+1}_{(x_0,y_0)}\left( (x_1,y_1),(x_2,y_2) \right )$ such that it satisfies the restrictions from $P_{n+1}(x_0,y_0)$. Since we know the condition $y_2\geq y_1$, it would be natural to send $$f^{n+1}_{(x_0,y_0)}((x_1,y_1),(x_2,y_2))_1=y_2-y_1$$ which would ensure us the necessary condition for positivity.\\
Consider $f^{n+1}_{(x_0,y_0)}$ to be defined by the following formula:
$$f^{n+1}_{(x_0,y_0)}\left( (x_1,y_1),(x_2,y_2) \right)=\left(y_2-y_1,x_2 \right).$$ 
Then $f^{n+1}_{(x_0,y_0)}\left( (0,0),(0,0)\right)=(0,0)$, so it preserves the initial points. Now we can check that this transformation, preserves correspondingly all the possible movements that we described in the two cases, in the following way:\label{rules}
$$ \left((x_1,y_1), \ (x_2,y_2) \right) \longrightarrow \left( y_2-y_1,x_2 \right)$$
$$(0,1), (0,1) \ \ \ \ \ \ \ \longleftrightarrow \ \ \ \ \ \ \ \ (0,0)$$
$$(1,0), \ (1,0) \ \ \ \ \ \ \ \longleftrightarrow \ \ \ \ \ \ \ \ (0,1)$$
$$(1,0), \ (0,1) \ \ \ \ \ \ \ \longleftrightarrow \ \ \ \ \ \ \ \ (1,0)$$
$$ \ (0,1), \ (1,0) \ \ \ \ \ \ \ \longleftrightarrow \ \ \ \ \ \ \ (-1,1)$$
This shows that the function $f^{n+1}_{(x_0,y_0)}$ is a well-defined bijection between ${D_{n}^\Delta }(x_0+y_0,y_0)$ and $P_{n+1}(x_0,y_0)$ and concludes the proof.
\end{proof}
Using the previous combinatorial interpretations, we prove Theorem \ref{conj}:
$$dim(LG_{n+1}(\alpha))=\frac{(2n)!(2n+1)!}{(n!(n+1)!)^2}.$$
\begin{proof}
$$dim(LG_{n+1}(\alpha))=^{Cor \ref{1}} \sum_{x,y\in \Delta_{n}}|P_{n+1}(x,y)|^2 =^{Thm \ref{3}} \sum_{x,y\in \Delta_{n}}|D^{\Delta}_{n+1}(x+y,y)|^2=$$
$$=^{(a'=x+y, b'=y)}\sum_{a',b'\in \bar{\Delta}_{n}}|D^{\Delta}_{n+1}(a',b')|^2=^{Lemma \ref{C}} |C_{n+1} =^{Thm \ref{2}}\frac{(2n)!(2n+1)!}{(n!(n+1)!)^2}.$$
\end{proof}
\section{Algorithm for the combinatorial computation}\label{sec7'}
In this part, we will review the main steps from above and see explicitely how starting with a simple component associated to a path from $T_{n+1}(x,y)$, we can identify it with a certain pair of paths in $C_{n+1}(x+y+1,y)$.
\begin{remark}
Let $V\big( x, (n+1)\alpha+y) \big)$ be any simple component from the decomposition of $V(0,\alpha)^{n+1}$, for a point $(x,y)\in \Delta_{n}$ (we know that all irreducible modules from $V(0,\alpha)^{\otimes n}$ have this form). Then, this module will be one of the simple modules that occur in the isotypic component 
$$T_{n+1}(x,y)\otimes V(0,\alpha).$$ 
In other words, this simple component is encoded by a pair of the following form:
$$\big( (x,y)\in \Delta_n,  \tau \in T_{n+1}(x,y) \big)$$ 
\end{remark}

\begin{remark}
From the identifications that are used in the previous section (\ref{CD},\ref{fff},\ref{ggg}), we see that for any point in the simplex $ (x,y) \in \Delta_{n}$ there is a one-to-one correspondence between the multiplicity space $T_{n+1}(x,y)$ associated to it and the set of pair of paths which cut the diagonal into two fixed points:
$$\mathscr F^{n+1}_{(x,y)}:T_{n+1}(x,y) \rightarrow C^{\Delta}_{n+1}(x+y+1,y) $$
$$\mathscr F^{n+1}_{(x,y)}={\big(\eta^{n+1}_{(x+y+1,y)}\big)^{-1}}\circ \big(f^{n+1}_{(x,y)}\big)^{-1}\circ \psi^{n+1}_{(x,y)}.$$
\end{remark}
We summarise all the parts that we have done so far in the following diagram: 
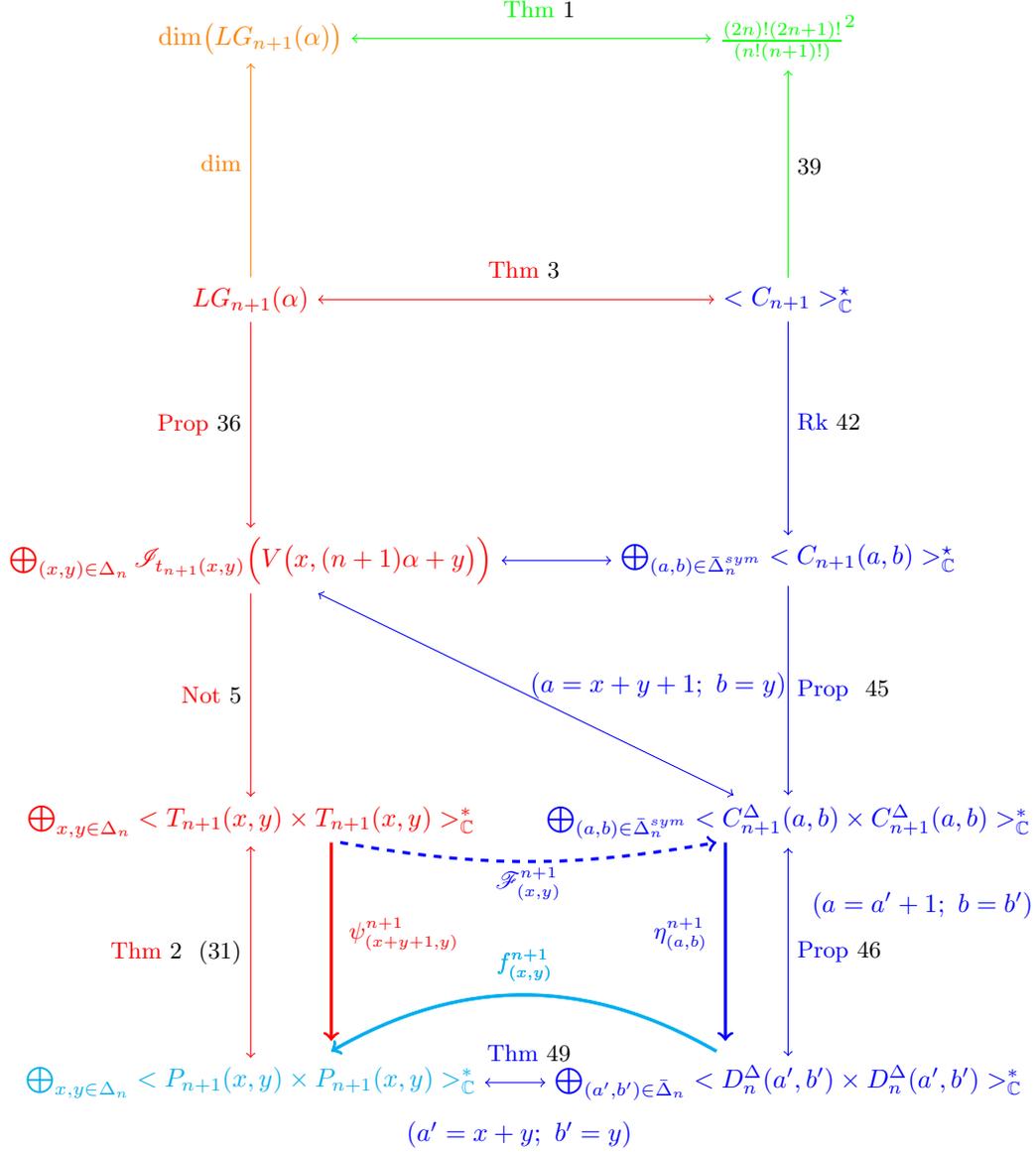
\begin{center} 
\begin{tikzpicture}
[x=1.2mm,y=1.4mm]

\node (b1)  [color=cyan]  at (-30,0)    {$ \bigoplus_{x,y \in \Delta_{n}} <P_{n+1}(x,y) \times P_{n+1}(x,y)>^{*}_{\C}$};
\node (b2) [color=red] at (-30,25)   {$\bigoplus_{x,y \in \Delta_{n}} <T_{n+1}(x,y) \times T_{n+1}(x,y)>^{*}_{\C}$};
\node (b3) [color=red] at (-30,50)   {$\bigoplus_{(x,y)\in \Delta_{n}} \mathscr{I}_{t_{n+1}(x,y)}\Big(V\big(x,(n+1)\alpha+y\big)\Big)$};
\node (b4) [color=red] at (-30,75)   {$LG_{n+1}(\alpha)$};
\node (b5) [color=orange] at (-30,100)   {$\text{dim} \big( LG_{n+1}(\alpha)\big) $};

\node (t1)  [color=blue]             at (30,0)   {$ \bigoplus_{(a',b') \in \bar{\Delta}_{n}} <D^{\Delta}_{n}(a',b') \times D^{\Delta}_{n}(a',b')>_{\C}^{*}$};
\node (t2) [color=blue] at (30,25)  {$\bigoplus_{(a,b) \in \bar{\Delta}^{sym}_{n}} <C^{\Delta}_{n+1}(a,b) \times C^{\Delta}_{n+1}(a,b)>^{*}_{\C}$};
\node (t3) [color=blue] at (30,50)  {$\bigoplus_{(a,b) \in \bar{\Delta}^{sym}_{n}} <C_{n+1}(a,b)>^{\star}_{\C}$};
\node (t4) [color=blue] at (30,75)  {$<C_{n+1}>^{\star}_{\C}$};
\node (t5) [color=green] at (30,100)  {$\frac{(2n)!(2n+1)!}{(n!(n+1)!)}^2$};
\node  [color=blue] at (0,-5)  {$(a'=x+y; \ b'=y)$};
\node  [color=blue] at (45,17)  {$(a=a'+1; \ b=b')$};
\node  [color=blue] at (15.5,38)  {$(a=x+y+1; \ b=y)$};
\draw[<->,color=blue]             (b1)      to node[right,yshift=4mm,xshift=-5mm,font=\small]{$\text{Thm} \ \ref{3}$}                           (t1);
\draw[<->,color=red]   (b1)      to node[left,font=\small]{$\text{Thm} \ \ref{Thm1} \ \ (\ref{set})$}                           (b2);
\draw[<-,color=red]   (b2)      to node[left,font=\small]{$\text{Not} \ \ref{mult}$}                           (b3);
\draw[<-,color=red]   (b3)      to node[left,font=\small]{$\text{Prop} \ \ref{inter}$}                           (b4);

\draw[<-,color=blue]   (t2)      to node[right,font=\small]{ \text{Prop} \ \ref{splitpairs}}                           (t3);

\draw[<->,color=blue]  (t1)      to node[right,font=\small]{$\text{Prop} \ \ref{CD}$}                           (t2);
\draw[->, color=green]             (t4)      to node[right,yshift=1mm,font=\small]{$\ref{2}$}   ( t5);
\draw[<-, color=blue]             (t3)      to node[right,yshift=1mm,font=\small]{$\text{Rk} \ \ref{disjunion}$}   ( t4);
\draw[->, color=orange]             (b4)      to node[left,yshift=1mm,font=\small]{$\text{dim}$}   ( b5);
\draw[<->, color=red]             (b4)     to node[right,yshift=4mm,xshift=-5mm,font=\small]{$ \text{Thm} \ \ref{Thm3}$}   ( t4);
\draw[<->, color=green]             (b5)     to node[right,yshift=4mm,xshift=-5mm,font=\small]{$ \text{Thm} \ \ref{conj}$}   ( t5);
\draw[<->,color=blue]  (t3)      to node[right,font=\small]{}                           (b3);
\draw[<->,color=blue]  (b3)      to node[right,font=\small]{}                           (t2);
\draw[<-, color=cyan, very thick]             (-21,3)     to[in=150,out=30] node[right,yshift=4mm,xshift=-5mm,font=\small]{$f^{n+1}_{(x,y)}$}   (22,3);
\draw[<-, color=cyan, very thick]             (-21,3)     to[in=150,out=30] node[right,yshift=4mm,xshift=-5mm,font=\small]{$f^{n+1}_{(x,y)}$}   (22,3);
\draw[->, color=red, very thick]             (-21,23)     to node[right,xshift=-1mm,yshift=1mm,xshift=2mm,font=\small]{$\psi^{n+1}_{(x+y+1,y)}$}  (-21,4);
\draw[->, color=blue, very thick]             (23,23)     to node[left,yshift=1mm,xshift=-1mm,font=\small]{$\eta^{n+1}_{(a,b)}$}  (23,4);
\draw[->, color=blue, very thick, dashed]             (-20,23)     to[in=-170,out=-10] node[yshift=-3mm,font=\small]{$\mathscr F^{n+1}_{(x,y)}$}  (22,23);

\end{tikzpicture}
\end{center}
\subsection{Description of $\mathscr F^{n+1}_{(x,y)}(\tau)$}
Let $(x,y) \in \Delta_n$ and a path $\tau\in T_{n+1}(x,y)$. This path can be encoded by a sequence of moves $$\tau=(\tau_1,....,\tau_n), \ \ \tau_{i} \in \{M_1,M_2,M_3,M_4 \}.$$
Let us denote $$\mathscr F^{n+1}_{(x,y)}(\tau)_i=(\sigma_i^0,...,\sigma_i^{n+1})$$
We begin with $\sigma_1^0=(0,1);  \sigma_2^0=(1,0)$. Then, each move from $\tau$ will prescribe the construction of a new double point in $\sigma$ following the rules from \ref{rules}:
$$\tau_i \leadsto \big( (\sigma_1^i, \sigma_2^i)\rightarrow(\sigma_1^{i+1}, \sigma_2^{i+1})\big)$$ 
More precisely, each move from the path $\tau$ will be replaced by two segments in the following manner:
\begin{center}
\begin{equation}\label{eq:bla5}
\begin{split}
\begin{tikzpicture}
[x=1mm,y=1mm,scale=3/4,font=\Large]
\node at (-60,25) [anchor=north east] {$\Large{M2)}$};
\node at (-60,55) [anchor=north east] {$\Large{M1)}$};
\node at (20,25) [anchor=north east] {$\Large{M4)}$};
\node at (20,55) [anchor=north east] {$\Large{M3)}$};
\node at (-40,23) [anchor=north east] {$\longleftrightarrow$};
\node at (-40,53) [anchor=north east] {$\longleftrightarrow$};
\node at (50,23) [anchor=north east] {$\longleftrightarrow$};
\node at (50,53) [anchor=north east] {$\longleftrightarrow$};
\draw[->, black!20,thick,blue] (-57,15) -- (-57,25);
\foreach \x/\y in {-57/52} {\node at (\x,\y) [circle,fill,inner sep=1pt] {};}
\draw[->, black!20,thick,blue] (20,51) -- (30,51);
\draw[->, black!20,thick,blue] (30,20) -- (20,30);
\draw[->, black!20,thick,red] (-30,50) -- (-30,60);
\draw[->, black!20,thick,green] (-15,40) -- (-15,50);
\draw[->, black!20,thick,red] (-30,25) -- (-20,25);
\draw[->, black!20,thick,green] (-15,15) -- (-5,15);
\draw[->, black!20,thick,red] (60,50) -- (60,60);
\draw[->, black!20,thick,green] (75,40) -- (85,40);
\draw[->, black!20,thick,red] (60,25) -- (70,25);
\draw[->, black!20,thick,green] (75,15) -- (75,25);
\node at (-5,45) [anchor=north east] {$\color{green}\Large{\sigma_1}$};
\node at (-20,55) [anchor=north east] {$\color{red}\Large{\sigma_2}$};
\node at (-5,15) [anchor=north east] {$\color{green}\Large{\sigma_1}$};
\node at (-20,20) [anchor=north east] {$\color{red}\Large{\sigma_2}$};
\node at (82,47) [anchor=north east] {$\color{green}\Large{\sigma_1}$};
\node at (70,55) [anchor=north east] {$\color{red}\Large{\sigma_2}$};
\node at (85,23) [anchor=north east] {$\color{green}\Large{\sigma_1}$};
\node at (70,20) [anchor=north east] {$\color{red}\Large{\sigma_2}$};
\end{tikzpicture}
\end{split}
\end{equation}
\end{center}
One concludes the folowing correspondence between simple components and pairs of planar paths, as in the figure \ref{computation}:
$$V \big(x, (n+1)\alpha+y \big) \longleftrightarrow \big(  (x,y) \in \Delta_n, \tau\in T_{n+1}(x,y) \big)\longleftrightarrow 
\big( {\mathscr F^{n+1}_{(x,y)}}(\tau)=(\sigma_1,\sigma_2)\big) \in C^{\Delta}_{n+1}(x+y+1,y)
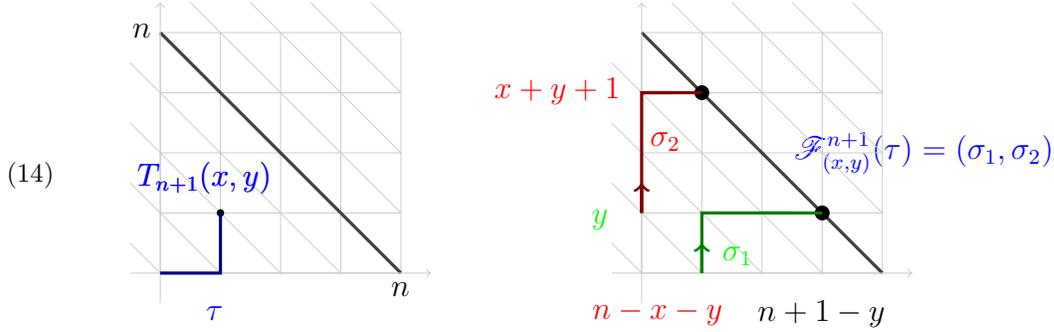
$$
\begin{center}
\begin{equation}\label{computation}
\begin{split}
\begin{tikzpicture}
[x=1mm,y=1mm,scale=4/5,font=\Large]
\draw[step=10mm,black!20,very thin] (0,0) grid (40,40);
\draw[black!20,very thin] (0,0) -- +(-5,5);
\draw[black!20,very thin] (10,0) -- +(-15,15);
\draw[black!20,very thin] (20,0) -- +(-25,25);
\draw[black!20,very thin] (30,0) -- +(-35,35);
\draw[very thick,black!50!gray] (40,0) -- +(-40,40);
\draw[black!20,very thin] (40,10) -- +(-35,35);
\draw[black!20,very thin] (40,20) -- +(-25,25);
\draw[black!20,very thin] (40,30) -- +(-15,15);
\draw[black!20,very thin] (40,40) -- +(-5,5);
\draw[->,black!20] (0,-5) -- (0,45);
\draw[->,black!20] (-5,0) -- (45,0);
\node at (12,-4) [anchor=north east] {$\Large{\color{blue}\tau}$};
\node at (20,20) [anchor=north east] {$\Large{\color{blue}T_{n+1}(x,y)}$};
\node at (20,20) [anchor=north east] {$\Large{\color{blue}T_{n+1}(x,y)}$};
\node at (43,0) [anchor=north east] {$\Large{n}$};
\node at (0,43) [anchor=north east] {$\Large{n}$};
\foreach \x/\y in {10/10} {\node at (\x,\y) [circle,fill,inner sep=1pt] {};}
\draw[very thick,black!50!blue] (0,0) -- (10,0) -- (10,10); 
\node at (88,25) [anchor=north east] {$\Large{\color{red}\sigma_2}$};
\node at (100,6) [anchor=north east] {$\Large{\color{green}\sigma_1}$};
\node at (78,34) [anchor=north east] {$\Large{\color{red}x+y+1}$};
\node at (76,12) [anchor=north east] {$\Large{\color{green}y}$};
\node at (95,-3) [anchor=north east] {$\Large{\color{red}n-x-y}$};
\node at (122,-3) [anchor=north east] {$\Large{n+1-y}$};
\node at (150,25) [anchor=north east] {$\color{blue}\mathscr F^{n+1}_{(x,y)}(\tau)=(\sigma_1,\sigma_2)$};
\draw[step=10mm,black!20,very thin] (80,0) grid (120,40);
\draw[black!20,very thin] (80,0) -- +(-5,5);
\draw[black!20,very thin] (90,0) -- +(-15,15);
\draw[black!20,very thin] (100,0) -- +(-25,25);
\draw[black!20,very thin] (110,0) -- +(-35,35);
\draw[very thick,black!50!gray] (120,0) -- +(-40,40);
\draw[black!20,very thin] (120,10) -- +(-35,35);
\draw[black!20,very thin] (120,20) -- +(-25,25);
\draw[black!20,very thin] (120,30) -- +(-15,15);
\draw[black!20,very thin] (120,40) -- +(-5,5);
\draw[->,black!20] (80,-5) -- (80,45);
\draw[->,black!20] (75,0) -- (125,0);
\foreach \x/\y in {90/30, 110/10} {\node at (\x,\y) [circle,fill,inner sep=2pt]{};}
\draw[very thick,black!50!green] (90,0) -- (90,10) -- (110,10); 
\draw[very thick,black!50!green,->] (90,0) -- (90,5);
\draw[very thick,black!50!green] (100,10) -- (105,10);
\draw[very thick,black!50!red] (80,10) -- (80,30) -- (90,30);
\draw[very thick,black!50!red,->] (80,10) -- (80,15);
\draw[very thick,black!50!red] (80,30) -- (85,30);
\end{tikzpicture}
\end{split}
\end{equation}
\end{center}

\section{Basis of matrix units for $LG_{n+1}$ using planar curves}{\label{Sec7}}
 
In this part, the aim is to describe a matrix unit basis for the centraliser algebra $LG_{n+1}(\alpha)$. As we have seen, the Links Gould centraliser algebra is a semi-simple algebra. We make the correspondence between each simple component of $LG_{n+1}(\alpha)$ and an algebra of matrices indexed by planar curves in the square, which cut the second diagonal into two fixed points.  

We start by a brief summary of the results from the last two sections, where the discussion was focused on the dimension of the algebra $LG_{n+1}(\alpha)$, and we found different ways of describing it using paths in the plane. 

On the algebraic side, as discussed in section \ref{int}, the semi-simple structure of the tensor powers of the representation $V(0,\alpha)$ has the following form:
$$V(0,\alpha)^{\otimes n+1}=\bigoplus_{x,y \in \Delta_{n}}\Big( T_{n+1}(x,y) \otimes V\big(x,(n+1)\alpha+y \big) \ \Big),$$ 
where $T_{n+1}(x,y)$ is the multiplicity space corresponding to the weight $(x,(n+1) \alpha+y)$.
For any simple representation $V\big(x, (n+1)\alpha+y\big)$, where $(x,y) \in \Delta_n$, and any natural number $k \in \N$, the $k^{th}$ corresponding intertwinner space is denote by:
$$\mathscr{I}_{k}\Big(V\big(x, (n+1)\alpha+y\big)\Big)=End_{U_q(sl(2|1))}\left( \bigoplus_{i=1}^{k}  V_i \ \right)$$
 $$\text{  where  }  V_i \simeq V\big(x, (n+1)\alpha+y\big) \ \  \forall i \in \{1,...,k\}. $$
Since these representations are simple, each intertwinner space is isomorphic to an algebra of matrices:
$$\mathscr{I}_{n+1}(V(0,\alpha))\simeq M({n+1},\Bbbk).$$
This leads to the following matrix decomposition (\ref{inter}):
$$LG_{n+1}(\alpha)\simeq \bigoplus_{x,y \in \Delta_{n}} \mathscr{I}_{t_{n+1}(x,y)}\Big(V\big(x,(n+1)\alpha+y\big)\Big)\simeq \bigoplus_{x,y \in \Delta_{n}} M(t_{n+1}(x,y),\Bbbk).$$

On the combinatorial side, we will use the models from Theorem \ref{Thm1} and Theorem \ref{3}. Putting these correspondences together, we deduce that each intertwinner space that occur in $LG_{n+1}(\alpha)$ can be characterised by certain sets of pairs of planar curves. We remind the notations that we have used in the definition \ref{disjunion}.
For $(a,b)\in \bar{\Delta}^{sym}_n$, let us consider the following sets:
$$C_{n+1}(a,b)=\{\text {pairs of disjoint paths in the square of edge } n+1, 
\text { constructed with moves } M_2 \text{ or } M_3,$$
$$ \ \ \ \ \ \ \ \ \ \ \ \ \ \ \ \ \ \ \ \ \text{ between the points }(0,1)\rightarrow (n,n+1) \text{ and } (1,0) \rightarrow (n+1,n), \text{ that cut the diagonal }$$
$$ \text{ of the square} \text{ precisely in the points} \  (n+1-a,a) \ \text {and} \ (n+1-b,b) \} .$$
$${C_{n+1}^{\Delta}}(a,b)=\{ \text{pairs of disjoint paths in the simplex } \Delta _{n+1}, \text{ constructed with moves } M_2 \text { or } M_3,$$
$$ \ \ \ \ \ \ \ \ \ \ \ \ \ \ \ \ \ \ \ \ \ \ \text{ between the points } (1,0) \rightarrow (n+1-a,a) \text{ and } (0,1) \rightarrow (n+1-b,b) \text{ respectively} \} .$$
\begin{remark}\label{diagr}
There is the following one to one correspondence:
$$ \ \ \ \ \ \mathscr{I}_{t_{n+1}(x,y)}\Big(V\big(x,(n+1)\alpha+y\big)\Big) \simeq {<C_{n+1}(x+y+1,y)>}^{\star}_{\C} , \ \ \ \ \ \ \forall \ (x,y) \in \Delta_{n}$$
\end{remark}
\begin{proof}
Using the isomorphisms from Proposition \ref{inter}, Theorem \ref{Thm1}, Theorem\ref{Thm3}, Proposition \ref{CD} and Proposition \ref{splitpairs}, we obtain that for any point $ (x,y) \in \Delta_{n}$ there are the following correspondences:

\begin{center} 
\begin{equation} \label{iiiiiii}
\begin{tikzpicture}
[x=1.2mm,y=1.4mm]

\node (b1)  [color=cyan]  at (-30,0)    {$ P_{n+1}(x,y)$};
\node (b2) [color=red] at (-30,25)   {$ T_{n+1}(x,y)$};

\node (t1)  [color=blue]             at (30,0)   {$ D^{\Delta}_{n}(a',b')$};
\node (t2) [color=blue] at (30,25)  {$C^{\Delta}_{n+1}(a,b)$};
\node  [color=blue] at (0,-5)  {$(a'=x+y; \ b'=y)$};
\node  [color=blue] at (45,17)  {$(a=a'+1; \ b=b')$};
\draw[<->,color=blue]             (b1)      to node[right,yshift=4mm,xshift=-5mm,font=\small]{$\text{Thm} \ \ref{3}$}                           (t1);
\draw[<->,color=red]   (b1)      to node[left,font=\small]{$\text{Thm} \ \ref{Thm1} \ \ (\ref{set})$}                           (b2);
\draw[<->,color=blue]  (t1)      to node[right,font=\small]{$\text{Prop} \ \ref{CD}$}                           (t2);
\draw[<-, color=cyan, very thick]             (-21,3)     to[in=150,out=30] node[right,yshift=4mm,xshift=-5mm,font=\small]{$f^{n+1}_{(x,y)}$}   (22,3);
\draw[<-, color=cyan, very thick]             (-21,3)     to[in=150,out=30] node[right,yshift=4mm,xshift=-5mm,font=\small]{$f^{n+1}_{(x,y)}$}   (22,3);
\draw[->, color=red, very thick]             (-21,23)     to node[right,xshift=-1mm,yshift=1mm,xshift=2mm,font=\small]{$\psi^{n+1}_{(x+y+1,y)}$}  (-21,4);
\draw[->, color=blue, very thick]             (23,23)     to node[left,yshift=1mm,xshift=-1mm,font=\small]{$\eta^{n+1}_{(a,b)}$}  (23,4);
\draw[->, color=blue, very thick, dashed]             (-20,23)     to[in=-170,out=-10] node[yshift=-3mm,font=\small]{$\mathscr F^{n+1}_{(x,y)}$}  (22,23);
\end{tikzpicture}
\end{equation}
\end{center}
Putting everything together, we get:
$$ \ \ \ \ \ \mathscr{I}_{t_{n+1}(x,y)}\Big(V\big(x,(n+1)\alpha+y\big)\Big) \simeq^{\text{Thm}\ref{Thm1} \ (\ref{set})} {<P_{n+1}(x,y)\times P_{n+1}(x,y)>}^{\star}_{\C} \simeq $$
$$\simeq^{\text{Thm} \ \ref{3}}{<D^{\Delta}_{n+1}(x+y,y)\times D^{\Delta}_{n+1}(x+y,y)>}^{\star}_{\C}$$
$$\simeq^{Prop \ref{CD}}{<C^{\Delta}_{n+1}(x+y+1,y)\times C^{\Delta}_{n+1}(x+y+1,y)>}^{\star}_{\C} \simeq {<C_{n+1}(x+y+1,y)>}^{\star}_{\C}.$$
\end{proof}
In order to consider a basis for $LG_{n+1}(\alpha)$, we will use its intertwinner spaces. More specifically, for each intertwinner space, one has a natural basis, given by the projectors onto the corresponding simple components. We aim to describe a correspondence between this basis and a combinatorial model, constructed using sets of certain planar paths. In the sequel, we will introduce a new notation, which will codify more compactly the set of pairs of paths $C_{n+1}$, that occurs in the above characterisation.
\begin{notation}
Consider the following indexing sets:
$$S_{n+1}:= \{ \sigma \subseteq \R^2 \mid \sigma \text{ is a simple closed curve}$$
$$ \ \ \  \ \ \ \ \ \  \ \ \ \ \ \ \ \ \ \ \ \  \ \ \ \ \ \ \ \ \ \text { containing } \left(0,0\right) \text{ and } \ \left(n+1,n+1\right) $$
$$ \ \ \ \ \ \ \ \ \ \ \ \ \ \ \ \ \ \ \  \ \ \ \ \ \ \ \ \ \ \ \ \ \ \ {\text{ constructed with the moves } M_2 \text{ or } M_3 \}}^{\ast} .$$
$$S_{n+1}(a,b):=\{ \sigma \in S_{n+1} \mid \sigma \text{  cut the secondary diagonal in}$$
$$ \ \ \ \ \ \ \ \ \ \ \ \ \ \ \ \ \ \ \ \ \ \ \ \ \ \ \ \ \ \ \ \ \ \ \ \ \ \ \ \ \ \ \ \ \ \ \ \ \ \ \ \ \ \ \ \ \ \ \ {(n+1-a,a) \text{ and }  (n+1-b,b) \} }^{\ast}, \ \forall (a,b) \in \Delta^{sym}_{n}.$$
\end{notation}

We notice that if we start with a pair of paths in ${C_{n+1}} (a,b)$ and glue them at their ends, we obtain a curve in $S_{n+1}(a,b)$ (as in figure \ref{eq:bla4}). This shows that  for any point $(a,b)\in \bar{\Delta}^{sym}_n$, the family from above, which containes curves in the square, corresponds to the set of pairs of paths in the square:
$$S_{n+1}(a,b) \simeq {C_{n+1}} (a,b).$$
\begin{center}
\begin{equation}\label{eq:bla4}
\begin{split}
\begin{tikzpicture}
[x=1mm,y=1mm,scale=4/5,font=\Large]
\draw[step=10mm,black!20,very thin] (0,0) grid (40,40);
\draw[black!20,very thin] (0,0) -- +(-5,5);
\draw[black!20,very thin] (10,0) -- +(-15,15);
\draw[black!20,very thin] (20,0) -- +(-25,25);
\draw[black!20,very thin] (30,0) -- +(-35,35);
\draw[very thick,black!50!gray] (40,0) -- +(-40,40);
\draw[black!20,very thin] (40,10) -- +(-35,35);
\draw[black!20,very thin] (40,20) -- +(-25,25);
\draw[black!20,very thin] (40,30) -- +(-15,15);
\draw[black!20,very thin] (40,40) -- +(-5,5);
\draw[->,black!20] (0,-5) -- (0,45);
\draw[->,black!20] (-5,0) -- (45,0);
\node at (12,-4) [anchor=north east] {$\Large{\color{red}n+1-a}$};
\node at (-3,33) [anchor=north east] {$\Large{\color{red}a}$};
\node at (38,-3) [anchor=north east] {$\Large{\color{green}n+1-b}$};
\node at (-5,13) [anchor=north east] {$\Large{\color{green}b}$};
\foreach \x/\y in {0/10, 0/30, 10/0, 10/10, 30/10, 10/30} {\node at (\x,\y) [circle,fill,inner sep=1pt] {};}
\foreach \x/\y in {10/30, 30/10} {\node at (\x,\y) [circle,fill,inner sep=2pt]{};}
\draw[very thick,black!50!green] (10,0) -- (10,10) -- (30,10); 
\draw[very thick,black!50!green,->] (10,0) -- (10,5);
\draw[very thick,black!50!green] (20,10) -- (25,10);
\draw[very thick,black!50!red] (0,10) -- (0,30) -- (10,30);
\draw[very thick,black!50!red,->] (0,10) -- (0,15);
\draw[very thick,black!50!red] (0,30) -- (5,30);
\draw[very thick,black!50!green] (30,10) -- (30,30) -- (40,30);
\draw[very thick,black!50!green] (30,10) -- (30,15);
\draw[very thick,black!50!green,->] (30,30) -- (35,30);
\draw[very thick,black!50!red] (10,30) -- (20,30) -- (20,40) -- (30,40);
\draw[very thick,black!50!red] (10,30) -- (15,30);
\draw[very thick,black!50!red] (20,30) -- (20,35);
\draw[very thick,black!50!red,->] (20,40) -- (25,40);
\node at (92,-4) [anchor=north east] {$\Large{\color{red}n+1-a}$};
\node at (77,33) [anchor=north east] {$\Large{\color{red}a}$};
\node at (118,-3) [anchor=north east] {$\Large{\color{green}n+1-b}$};
\node at (75,13) [anchor=north east] {$\Large{\color{green}b}$};
\draw[step=10mm,black!20,very thin] (80,0) grid (120,40);
\draw[black!20,very thin] (80,0) -- +(-5,5);
\draw[black!20,very thin] (90,0) -- +(-15,15);
\draw[black!20,thick,orange] (90,0) -- +(-10,10);
\draw[black!20,thick,orange] (90,10) -- +(-10,10);
\draw[black!20,thick,orange] (100,10) -- +(-20,20);
\draw[black!20,thick,orange] (110,20) -- +(-10,10);
\draw[black!20,thick,orange] (110,30) -- +(-10,10);
\draw[black!20,thick,orange] (120,30) -- +(-10,10);
\draw[black!20,very thin] (100,0) -- +(-25,25);
\draw[black!20,very thin] (110,0) -- +(-35,35);
\draw[very thick,black!50!gray] (120,0) -- +(-40,40);
\draw[black!20,very thin] (120,10) -- +(-35,35);
\draw[black!20,very thin] (120,20) -- +(-25,25);
\draw[black!20,very thin] (120,30) -- +(-15,15);
\draw[black!20,very thin] (120,40) -- +(-5,5);
\draw[->,black!20] (80,-5) -- (80,45);
\draw[->,black!20] (75,0) -- (125,0);
\foreach \x/\y in {90/30, 100/30, 100/40, 110/40, 110/10, 110/30, 120/30} {\node at (\x,\y) [circle,fill,inner sep=1pt] {};}
\foreach \x/\y in {90/30, 110/10} {\node at (\x,\y) [circle,fill,inner sep=2pt]{};}
\draw[very thick,black!50!blue] (80,0) -- (90,0); 
\draw[very thick,black!50!blue] (80,0) -- (80,10); 
\draw[very thick,black!50!blue] (110,40) -- (120,40); 
\draw[very thick,black!50!blue] (120,30) -- (120,40); 
\draw[very thick,black!50!green] (90,0) -- (90,10) -- (110,10); 
\draw[very thick,black!50!green,->] (90,0) -- (90,5);
\draw[very thick,black!50!green] (100,10) -- (105,10);
\draw[very thick,black!50!red] (80,10) -- (80,30) -- (90,30);
\draw[very thick,black!50!red,->] (80,10) -- (80,15);
\draw[very thick,black!50!red] (80,30) -- (85,30);
\draw[very thick,black!50!green] (110,10) -- (110,30) -- (120,30);
\draw[very thick,black!50!green] (110,10) -- (110,15);
\draw[very thick,black!50!green,->] (110,30) -- (115,30);
\draw[very thick,black!50!red] (90,30) -- (100,30) -- (100,40) -- (110,40);
\draw[very thick,black!50!red] (90,30) -- (95,30);
\draw[very thick,black!50!red] (100,30) -- (100,35);
\draw[very thick,black!50!red,->] (100,40) -- (105,40);
\end{tikzpicture}
\end{split}
\end{equation}
$$ C_{n+1}(a,b) \ \ \ \ \ \ \  \ \ \ \ \ \ \ \ \ \ \ \ \ \ \ \ \ \ \ \ \ \ \ \ \  \ \ \ \ \ \ \ \ \ \ \ \ \ \ \ \ \ \ \ \ \ \ \ \ \ \  S_{n+1}(a,b) $$  
\end{center}

Following the ideas from section \ref{sec6}, we will pass to families of paths in the simplex rather than in the square. The advantage, as we have seen in diagram \ref{iiiiiii}, is the fact that this combinatorial data is a model for the isotypic components that occur in $V(0, \alpha)^{\otimes n}$.
\begin{remark}
We remind the following one to one correspondence:
$$C_{n+1}(a,b)\simeq C^{\Delta}_{n+1}(a,b)\times C^{\Delta}_{n+1}(a,b)$$ 
Using the previous notations, one has the following bijective map:
$$C^{\Delta}_{n+1}(a,b)\times C^{\Delta}_{n+1}(a,b)\longleftrightarrow S_{n+1}(a,b) $$
\end{remark}
Coming back to the algebra $LG_{n+1}(\alpha)$, we use the projections onto each simple module that occur in the semi-simple decomposition of the tensor powers of $V(0,\alpha)$.
\begin{definition}(Projectors onto simple components)\\
Let $n \in \N$. For a point $(x,y) \in \Delta_n$, let us consider the corresponding simple module $V\big(x,(n+1)\alpha+y\big)$. Then, its isotypic component inside $V(0, \alpha)^{\otimes {n+1}} $ is identified with:
$$T_{n+1}(x,y) \otimes V\big(x,(n+1)\alpha+y \big) \subseteq V(0, \alpha)^{\otimes n+1}.$$
Let $\tau \in T_{n+1}(x,y)$ a point in the multiplicity space.
We will denote the inclusion and projection onto the component $V\big(x,(n+1)\alpha+y \big)$ that corresponds to $\tau$ by the following maps:
$$\pi^{n+1}_{((x,y), \tau)}: V(0,\alpha)^{\otimes (n+1)}\rightarrow V\big(x,(n+1)\alpha+y\big)$$
$$\iota^{n+1}_{((x,y), \tau)}:V\big(x,(n+1)\alpha+y\big)\hookrightarrow V(0,\alpha)^{\otimes (n+1)}.$$
By composition, we obtain the element in $LG_{n+1}$, which is the projector onto the simple component of highest weight $(x, (n+1)\alpha+y)$, corresponding to the path $\tau$:
$$p^{n+1}_{((x,y),\tau)}:=\iota^{n+1}_{((x,y),\tau)} \circ \pi^{n+1}_{((x,y),\tau)}.$$
\end{definition}
Another class of interesting elements of the centraliser algebra $LG_{n+1}$ is the set of projectors onto isotypic components of $V(0,\alpha)^{n+1}$. They are idempotents and one can use them to describe part of the algebra $LG_{n+1 }$ by means of diagrams.
\begin{notation}(Projectors onto isotypic components)\\
1) For $(x,y)\in \Delta_n$, denote the projection and inclusion onto the corresponding isotypic component of the module $V\big(x,(n+1)\alpha+y \big)$ by:
$$\pi^{n+1}_{(x,y)}: V(0,\alpha)^{\otimes (n+1)}\rightarrow T_{n+1}\otimes V\big(x,(n+1)\alpha+y\big)$$
$$\iota^{n+1}_{(x,y)}:T_{n+1}\otimes V\big(x,(n+1)\alpha+y\big)\hookrightarrow V(0,\alpha)^{\otimes (n+1)}.$$
2) For any point $(x,y)\in \Delta_n$, consider $p^{n+1}_{x,y}\in LG_{n+1}(\alpha)$ to be the projector onto the isotypic component of $V\big(x,(n+1)\alpha+y\big)$, defined as:
$$p^{n+1}_{x,y}:=\iota^{n+1}_{(x,y)} \circ \pi^{n+1}_{(x,y)}.$$
\end{notation}
\begin{remark}
The projectors onto the isotypic components can be described using the projectors onto the corresponding simple components in the following manner:
$$p^{n+1}_{x,y}=\sum_{\tau \in T_{n+1}(x+y)}p^{n+1}_{((x,y),\tau)}.$$
\end{remark}
\begin{definition}
Let us denote the following set of symmetric curves:
$$S_{n+1}(x+y+1,y)^{ref}:= \{ \sigma \in S_{n+1}(x+y+1,y) \mid \sigma \text { is symmetric with respect to the diagonal of the square} \}.$$
\end{definition}
\begin{remark}
Using this set, we obtain the following bijection, which gives a model for the multiplicity spaces:
$$\phi: T_{n+1}(x,y)\rightarrow S_{n+1}^{ref}(x+y+1,y)$$
$$\phi(\tau)=\mathscr F^{n+1}_{(x,y)}(\tau)\sqcup \bar{\mathscr F}^{n+1}_{(x,y)} (\tau)\sqcup $$
$$ \ \ \ \ \ \ \ \ \ \ \ \ \ \ \ \ \ \ \ \ \  \ \ \ \ \ \  \sqcup \ [(0,0),(0,1)]\sqcup[(0,0),(1,0)]\sqcup$$
$$ \ \ \ \ \ \ \ \ \ \ \ \ \ \ \ \ \ \ \ \ \ \ \ \ \ \  \ \ \ \ \ \ \ \ \ \ \ \ \ \ \ \ \ \ \ \ \sqcup \ [(n,0),(n+1,0)]\sqcup[(n,n+1),(n+1,n+1)].$$
Here, $\bar{\mathscr F}^{n+1}_{(x,y)} (\tau)$ is the reflection of $\mathscr F^{n+1}_{(x,y)} (\tau)$ with respect to the second diagonal of the square.
\end{remark}
Putting everything together, we see that one can describe matrix unit basis for the centraliser algebra $LG_{n+1}(\alpha)$, using the space of elementary matrices indexed by simple closed curves in the plane, with integer coordinates. Moreover, in this model, the projectors corresponds to the space of curves that cut the diagonal of the square into two prescribed points.

\begin{theor}(\ref{Thm3})(Matrix unit basis for $LG_{n+1}(\alpha)$)\\
A basis for the algebra $LG_{n+1}$ is given by the following set of projectors: 
$$\mathscr B_{n+1}:=\{  p^{n+1}_{((x,y),\tau)} \mid (x,y)\in \Delta_n,  \tau \in T_{n+1}(x,y) \}$$
The basis $\mathscr B_{n+1}$ is into a one-to-one correspondence with the set of symmetric curves which cut the diagonal into the points $(n-x-y,x+y+1), (n+1-y,y)$:
$$\mathscr B_{n+1} \longleftrightarrow S^{ref}_{n+1}(x+y+1,y) $$ 
$$ p^{n+1}_{((x,y),\tau)}\longleftrightarrow \phi(\tau) \ \ \ \ \ \ \ \ \ \ \ \ \ \ \ \ $$
\end{theor}
\begin{remark}
Using this description of projectors, we conclude the following correspondences between simple components from the decomposition of tensor powers of $V(0,\alpha)$ and sets of symmetric curves in the plane, as in the figure \ref{ffffffffffi}:
$$ V(x, (n+1)\alpha+y) \ \ \longleftrightarrow \ \ \ \tau \in P_{n+1}(x,y)) \ \ \ \ \ \longleftrightarrow \ \ \ \ \ \ p^{n+1}_{((x,y),\tau)} \ \ \ \ \longleftrightarrow \ \ \ \ \ \ \phi(\tau) \in S^{ref}_{n+1}(x+y+1,y)$$  
\end{remark}
\begin{center}
\begin{equation}\label{ffffffffffi}
\begin{split}
\begin{tikzpicture}
[x=1mm,y=1mm,scale=4/5,font=\Large]
\draw[step=10mm,black!20,very thin] (0,0) grid (40,40);
\draw[black!20,very thin] (0,0) -- +(-5,5);
\draw[black!20,very thin] (10,0) -- +(-15,15);
\draw[black!20,very thin] (20,0) -- +(-25,25);
\draw[black!20,very thin] (30,0) -- +(-35,35);
\draw[very thick,black!50!gray] (40,0) -- +(-40,40);
\draw[black!20,very thin] (40,10) -- +(-35,35);
\draw[black!20,very thin] (40,20) -- +(-25,25);
\draw[black!20,very thin] (40,30) -- +(-15,15);
\draw[black!20,very thin] (40,40) -- +(-5,5);
\draw[->,black!20] (0,-5) -- (0,45);
\draw[->,black!20] (-5,0) -- (45,0);
\node at (12,-4) [anchor=north east] {$\Large{\color{blue}\tau}$};
\node at (20,20) [anchor=north east] {$\Large{\color{blue}T_{n+1}(x,y)}$};
\node at (20,20) [anchor=north east] {$\Large{\color{blue}T_{n+1}(x,y)}$};
\node at (43,0) [anchor=north east] {$\Large{n}$};
\node at (0,43) [anchor=north east] {$\Large{n}$};
\foreach \x/\y in {10/10} {\node at (\x,\y) [circle,fill,inner sep=1pt] {};}
\draw[very thick,black!50!blue] (0,0) -- (10,0) -- (10,10); 
\node at (97,21) [anchor=north east] {$\Large{\color{blue}\mathscr F(\tau)}$};
\node at (109,37) [anchor=north east] {$\Large{\color{blue}\bar{\mathscr F}(\tau)}$};
\node at (133,27) [anchor=north east] {$\Large{\color{blue}\phi(\tau)}$};
\node at (78,34) [anchor=north east] {$\Large{\color{red}x+y+1}$};
\node at (76,12) [anchor=north east] {$\Large{\color{green}y}$};
\node at (95,-3) [anchor=north east] {$\Large{\color{red}n-x-y}$};
\node at (122,-3) [anchor=north east] {$\Large{n+1-y}$};
\draw[step=10mm,black!20,very thin] (80,0) grid (120,40);
\draw[black!20,very thin] (80,0) -- +(-5,5);
\draw[black!20,very thin] (90,0) -- +(-15,15);
\draw[black!20,very thin] (100,0) -- +(-25,25);
\draw[black!20,very thin] (110,0) -- +(-35,35);
\draw[very thick,black!50!gray] (120,0) -- +(-40,40);
\draw[black!20,very thin] (120,10) -- +(-35,35);
\draw[black!20,very thin] (120,20) -- +(-25,25);
\draw[black!20,very thin] (120,30) -- +(-15,15);
\draw[black!20,very thin] (120,40) -- +(-5,5);
\draw[->,black!20] (80,-5) -- (80,45);
\draw[->,black!20] (75,0) -- (125,0);
\foreach \x/\y in {90/30, 110/10} {\node at (\x,\y) [circle,fill,inner sep=2pt]{};}
\draw[very thick,black!50!green] (90,0) -- (90,10) -- (110,10); 
\draw[very thick,black!50!green,->] (90,0) -- (90,5);
\draw[very thick,black!50!green] (100,10) -- (105,10);
\draw[very thick,black!50!red] (80,10) -- (80,30) -- (90,30);
\draw[very thick,black!50!red,->] (80,10) -- (80,15);
\draw[very thick,black!50!red] (80,30) -- (85,30);
\draw[very thick,black!50!blue] (80,0) -- (90,0); 
\draw[very thick,black!50!blue] (80,0) -- (80,10); 
\draw[very thick,black!50!blue] (110,40) -- (120,40); 
\draw[very thick,black!50!blue] (120,30) -- (120,40); 
\draw[very thick,black!50!green] (90,0) -- (90,10) -- (110,10); 
\draw[very thick,black!50!green,->] (90,0) -- (90,5);
\draw[very thick,black!50!green] (100,10) -- (105,10);
\draw[very thick,black!50!red] (80,10) -- (80,30) -- (90,30);
\draw[very thick,black!50!red,->] (80,10) -- (80,15);
\draw[very thick,black!50!red] (80,30) -- (85,30);
\draw[very thick,black!50!green] (110,10) -- (110,30) -- (120,30);
\draw[very thick,black!50!green] (110,10) -- (110,15);
\draw[very thick,black!50!green,->] (110,30) -- (115,30);
\draw[very thick,black!50!red] (90,30) -- (90,40) -- (100,40) -- (110,40);
\draw[very thick,black!50!red,->] (100,40) -- (105,40);
\end{tikzpicture}
\end{split}
\end{equation}
\end{center}
\begin{corollary}(Combinatorial model for the projectors onto isotypic components)\\
The projectors onto the isotypic components are in one-to-one correspondence with the following set of symmetric curves in the square:
\begin{equation}\label{projjj}
p^{n+1}_{x,y} \ \ \ \ \ \ \ \longleftrightarrow \sum_{\sigma \in S_{n+1}^{ref}(x+y+1,y)} \sigma^{\star}
\end{equation}
\end{corollary}
\begin{corollary}
a) The intertwinner spaces are identified with spaces generated by curves which cut the diagonal in two fixed points:
$$\mathscr{I}_{t_{n+1}(x,y)}\Big(V\big(x,(n+1)\alpha+y\big)\Big) \simeq {<S_{n+1}(x+y+1,y)>_{\C}}^{\star}, \ \ \forall (x,y)\in \Delta_n.$$
b)The endomorphism algebra has the following combinatorial description:
$$LG_{n+1}(\alpha)\simeq {<S_{n+1}>_{\C}}^{\star}.$$
\end{corollary}

\

}

\
\
\url{https://www.maths.ox.ac.uk/people/cristina.palmer-anghel}  


\begin{thebibliography}{99}
\bibitem {A} M. Aigner-{\em A Course in Enumeration}, Graduate Texts in Mathematics,Volume 238, 2007
\bibitem {BB} A. Beliakova, C. Blanchet, Skein construction of idempotents in the Birman-Murakami-Wenzl algebras . Math Ann 321 (2001) 2, 347-373.
\bibitem{BW} J. Barrett, B.  Westbury - {\em Spherical categories.}  Adv. Math.  \textbf{143} (1999), 357--375.
\bibitem{C1} A.M. Cohen, B. Frenk, D.B. Wales- {\em Brauer algebras of simply laced type}, Israel Journal of Mathematics, 173 (2009) 335–365.
\bibitem{C2} Arjeh M. Cohen, D. A.H. Gijsbers, D. B. Wales- {\em BMW algebras of simply laced type}, Journal of Algebra 286 (2005) 107–153.
\bibitem{C3}A. M. Cohen, D. A. H. Gijsbers, D. B. Wales- {\em The Birman-Murakami-Wenzl algebras of type Dn}, Communications in Algebra, 42 (2014), 22-55.
\bibitem{C4} A. M. Cohen, S. Liu, S. Yu-{ \em Brauer algebras of type C}, Journal of Pure and Applied Algebra 216 (2012) 407–426.
\bibitem{DKL} D. De Wit, L.H. Kauffman, J. Links-{\em On the Links-Gould invariant of links. J. Knot Theory Ramifications. 8 (1999), no.2, 165-199. MR 2000j:57020}
\bibitem{GP} N. Geer, B. Patureau-Mirand, {\em  Multivariable link invariants arising from sl(2,1) and the Alexander polynomial},  Journal of Pure and Applied Algebra , 210 (2007), no. 1, 283-298. 
\bibitem{GP1} N. Geer, B. Patureau-Mirand,-{\em On the colored HOMFLY-PT, multivariable and Kashaev link invariants }Communications in Contemporary Mathematics Vol. 10, Suppl. 1 (2008) 993–1011
\bibitem{I} A. Ishii-{\em The Links–Gould invariant of closed 3-braids}, J. Knot Theory Ramifications 13 (2004) 41–56.
\bibitem{I2}A. Ishii -{ \em The Links-Gould Polynomial as a Gener- alization of the Alexander-Conway Polynomial},  Pacific J. Math. 225 (2006), 273–285.
\bibitem{KT} Khoroshkin, S. M. Tolstoy, {\em Universal $R$-matrix for quantized (super)algebras}, Comm. Math. Phys. 141 (1991), no. 3, 599--617.
\bibitem{K} B.M Kohli-{\em On the Links–Gould invariant and the square of the Alexander polynomial	}, Journal of Knot Theory and Its Ramifications, Vol. 25, No. 2, (2016), 1650006.
\bibitem{LeZa} G. Lehrer, R. Zhang- {\em  Strongly multiplicity free modules for Lie algebras and quantum groups} , J. Algebra 306 (2006), 138-174.
\bibitem{LG} J.R. Links, M.D. Gould- {\em Two variable link polynomials from quantum supergroups}, Lett. Math. Phys. 26 (1992)
187–198.
\bibitem{MW}  I. Marin, E. Wagner- {\em A cubic defining algebra for the Links-Gould polynomial}, Adv. Math. 248 (2013), 1332-1365
\bibitem{Mu} J. Murakami-{\em The Kaufman polynomial of links and representation theory}, Osaka J. Math., 24:745–
758, (1987).
\bibitem{Mu2} J. Murakami-{ \em The representations of the q-analogue of Brauer's Centralizer Algebras and the Kauffman Polynomial of Links}, Publ. RIMS, Kyoto Univ.
26 (1990), 935-945.
\bibitem{Mu3} M. Kosuda, J. Murakami-{ \em Centralizer algebras of the mixed tensor representations of quantum group $U_q(gl(\pi, \C))^{*}$  } J. Osaka J. Math.
30 (1993), 475-507.
\bibitem{RW}  A. Ram, H. Wenzl-{\em Matrix units for centralizer algebras}, J. of Algebra. Vol. 145 (1992), 378–395.
\bibitem {S} R. Stanley- {\em Enumerative combinatorics}, Volume 2, Cambridge 
University Press, (1999). 
\bibitem{Wz} {H. Wenzl}-{\em On the structure of Brauer's centralizer algebras}, Ann. of Math. 128 (1988), 179-193.
\bibitem{Y} {H Yamane}, {\em Quantized enveloping algebras associated with simple Lie superalgebras and their universal $R$-matrices},  Publ. Res. Inst. Math. Sci.  30 (1994),  no. 1, 15--87.  

\end{thebibliography}
\end{document}